\documentclass[english]{article}
\usepackage[T1]{fontenc}
\usepackage[latin9]{inputenc}
\usepackage{babel}
\usepackage{varioref}
\usepackage{prettyref}
\usepackage{float}
\usepackage{amsthm}
\usepackage{amsmath}
\usepackage{amssymb}
\usepackage{mathdots}
\usepackage{graphicx}
\usepackage{esint}
\usepackage{xargs}[2008/03/08]
\PassOptionsToPackage{version=3}{mhchem}
\usepackage{mhchem}
\usepackage[unicode=true,pdfusetitle,
 bookmarks=true,bookmarksnumbered=false,bookmarksopen=false,
 breaklinks=false,pdfborder={0 0 1},backref=false,colorlinks=false]
 {hyperref}
\usepackage{breakurl}

\makeatletter

\providecommand{\tabularnewline}{\\}

\floatstyle{ruled}
\newfloat{algorithm}{tbp}{loa}
\providecommand{\algorithmname}{Algorithm}
\floatname{algorithm}{\protect\algorithmname}

\numberwithin{equation}{section}
\theoremstyle{plain}
\newtheorem{thm}{\protect\theoremname}[section]
  \theoremstyle{definition}
  \newtheorem{problem}[thm]{\protect\problemname}
  \theoremstyle{definition}
  \newtheorem{defn}[thm]{\protect\definitionname}
  \theoremstyle{remark}
  \newtheorem{rem}[thm]{\protect\remarkname}
  \theoremstyle{plain}
  \newtheorem{lem}[thm]{\protect\lemmaname}
  \theoremstyle{plain}
  \newtheorem{cor}[thm]{\protect\corollaryname}
  \theoremstyle{definition}
  \newtheorem{example}[thm]{\protect\examplename}
  \theoremstyle{plain}
  \newtheorem{prop}[thm]{\protect\propositionname}
 \theoremstyle{definition}
 \newtheorem*{defn*}{\protect\definitionname}
  \theoremstyle{plain}
  \newtheorem*{thm*}{\protect\theoremname}

\usepackage{etex}

\let\iint\@undefined
\let\iiint\@undefined
\let\iiiint\@undefined
\let\diag\@undefined
\let\*\@undefined

\usepackage{dimadefs}

\usepackage{fullpage}

\newcommand{\ff}[2]{\ensuremath{(#1)_{#2}}} 

\DeclareMathOperator{\diag}{diag}

\DeclareMathOperator{\supp}{supp}

\usepackage{prettyref}
\newrefformat{sub}{Subsection \ref{#1}}
\newrefformat{app}{Appendix \ref{#1}}
\newrefformat{cor}{Corollary \ref{#1}}
\newrefformat{alg}{Algorithm \ref{#1}}
\newrefformat{def}{Definition \ref{#1}}
\newrefformat{clm}{Claim \ref{#1}}
\newrefformat{prop}{Proposition \ref{#1}}
\newrefformat{conj}{Conjecture \ref{#1}}
\newrefformat{prob}{Problem \ref{#1}}
\newrefformat{step}{step \ref{#1}}
\newrefformat{rem}{Remark \ref{#1}}
\newrefformat{tbl}{Table \ref{#1}}

\renewcommand{\vec}[1]{\ensuremath{{\boldsymbol #1}}}
\newcommand{\ivm}{\ensuremath{\mathcal{N}}} 

\newcommand{\RR}{\ensuremath{\mathbb{R}}}

\@ifundefined{showcaptionsetup}{}{%
 \PassOptionsToPackage{caption=false}{subfig}}
\usepackage{subfig}
\makeatother

  \providecommand{\corollaryname}{Corollary}
  \providecommand{\definitionname}{Definition}
  \providecommand{\examplename}{Example}
  \providecommand{\lemmaname}{Lemma}
  \providecommand{\problemname}{Problem}
  \providecommand{\propositionname}{Proposition}
  \providecommand{\remarkname}{Remark}
  \providecommand{\theoremname}{Theorem}
\providecommand{\theoremname}{Theorem}

\begin{document}

\title{Decimated generalized Prony systems\thanks{Department of Mathematics, Weizmann Institute of Science, Rehovot 76100, Israel.}}

\author{Dmitry Batenkov\thanks{\tt{dima.batenkov@weizmann.ac.il}. The author is supported by the Adams Fellowship Program of the Israel Academy of Sciences and Humanities.}}
\maketitle
\begin{abstract}
We continue studying robustness of solving algebraic systems of Prony
type (also known as the exponential fitting systems), which appear
prominently in many areas of mathematics, in particular modern ``sub-Nyquist''
sampling theories. We show that by considering these systems at arithmetic
progressions (or ``decimating'' them), one can achieve better performance
in the presence of noise. We also show that the corresponding lower
bounds are closely related to well-known estimates, obtained for similar
problems but in different contexts.
\end{abstract}

\section{Introduction}

\global\long\def\nd{\xi}
\global\long\def\jp{z}
\global\long\def\jc{a}
\global\long\def\np{\mathcal{K}}
\global\long\def\ord{d}
\newcommandx\meas[1][usedefault, addprefix=\global, 1=k]{m_{#1}}
\global\long\def\df{p}
\global\long\def\init{t}
\global\long\def\ns{N}
\newcommandx\er[1][usedefault, addprefix=\global, 1=k]{\delta_{#1}}
\global\long\def\ncoeffs{C}
\global\long\def\nparams{R}
\global\long\def\fun{f}
\global\long\def\nn#1{\widetilde{#1}}
\global\long\def\err{\varepsilon}
\global\long\def\cvand{V}
\global\long\def\cpvand{U}
\global\long\def\sc{\ns}

\global\long\def\vec#1{\ensuremath{\mathbf{#1}}}
\newcommandx\prow[3][usedefault, addprefix=\global, 1=\init, 2=\df]{\vec{u_{#3}^{\left(#1,#2\right)}}}
\newcommandx\crow[3][usedefault, addprefix=\global, 1=\init, 2=\df]{\vec{v_{#3}^{\left(#1,#2\right)}}}
\newcommandx\mult[2][usedefault, addprefix=\global, 1=r, 2=\init]{\ensuremath{Q_{#2,#1}}}
\global\long\def\cfmat{A}
\global\long\def\powermat#1#2{T_{#1,#2}}
\newcommandx\fwm[3][usedefault, addprefix=\global, 1=\init, 2=\df]{\mathcal{P}_{#1,#2}^{\left(#3\right)}}
\newcommandx\ccoeffbl[1][usedefault, addprefix=\global, 1=j]{\ensuremath{D_{#1}}}
\newcommandx\pcoeffbl[1][usedefault, addprefix=\global, 1=j]{\ensuremath{E_{#1}}}
\global\long\def\jac{\ensuremath{\mathcal{J}}}

\global\long\def\ppm{\fwm P}
\global\long\def\cpm{\fwm C}
\newcommandx\dm[3][usedefault, addprefix=\global, 1=\init, 2=\df]{\mathcal{M}_{#1,#2}^{\left(#3\right)}}
\global\long\def\pdm{\dm P}
\global\long\def\cdm{\dm C}
\global\long\def\stirlmat#1{\mathcal{S}_{#1}}
\global\long\def\sts#1#2{\left\{  \begin{array}{c}
 #1\\
#2 
\end{array}\right\}  }
\global\long\def\stf#1#2{\left[\begin{array}{c}
 #1\\
#2 
\end{array}\right]}
\global\long\def\jpbound{\Lambda}
\global\long\def\jcbound{A}
\global\long\def\accr{\mathcal{ACC}}
\global\long\def\laccr{\accr_{LOC}}
\global\long\def\noims{\tilde{\vec y}}
\global\long\def\exams{\vec y}
\global\long\def\src{\vec x}
\global\long\def\noisrc{\tilde{\src}}
\global\long\def\placcr{\laccr^{\left(\ppm\right)}}
\global\long\def\claccr{\laccr^{\left(\cpm\right)}}

The system of equations
\begin{equation}
\meas=\sum_{j=1}^{\np}\jc_{j}z_{j}^{k},\qquad\jc_{j},z_{j}\in\complexfield,\; k\in\naturals\label{eq:basic-prony}
\end{equation}
appeared originally in the work of Baron de Prony \cite{prony1795essai}
in the context of fitting a sum of exponentials to observed data samples.
He showed that the unknowns $\jc_{j},\jp_{j}$ can be recovered explicitly
from $\left\{ \meas[0],\dots,\meas[2\np-1]\right\} $ by what is known
today as ``Prony's method''. This ``Prony system'' appears in
areas such as frequency estimation, Padé approximation, array processing,
statistics, interpolation, quadrature, radar signal detection, error
correction codes, and many more. The literature on this subject is
huge (for instance, the bibliography on Prony's method from \cite{Auton1981}
is some 50+ pages long). Original Prony's solution method can be found
in many places, e.g. \cite{rao1992mbp}.

Our interest in \eqref{eq:basic-prony} and more general systems of
``Prony type'', defined in \prettyref{tbl:prony-systems} below,
originates from their central role in the so-called \emph{algebraic
sampling} approach to problems of signal reconstruction. The basic
idea there is to model the unknown signal by a small number of parameters,
and subsequently reconstruct these parameters from the given small
number of noisy measurements -- in fact, from a number which is much
smaller than would be required by the classical Shannon-Nyquist-Kotel'nikov-Whittaker
sampling theorem (hence the name ``sub-Nyquist sampling''). In many
cases of interest, Prony-type systems appear precisely as the equations
relating the parameters to the measurements. Let us give some examples.

Arguably the simplest sparse signal model, which is essentially non-bandlimited
(and therefore inaccessible to the classical sampling theory), is
given by a linear combination of a finite number of Dirac $\delta$-distributions
(``spikes'', or simply ``Diracs''):
\begin{equation}
f\left(x\right)=\sum_{j=1}^{\np}\jc_{j}\delta\left(x-\nd_{j}\right),\qquad\jc_{j}\in\reals,\;\nd_{j}\in\reals.\label{eq:sum-deltas}
\end{equation}

Such $f\left(x\right)$ is a useful model for many types of natural
signals, as well as in ranging and wideband communication \cite{kusuma2009accuracy}.
While Shannon's sampling theorem is inapplicable in this case (it
would require an infinitely fast sampling rate), it has been shown
in \cite{vetterli2002ssf} that such signals can be perfectly reconstructed
from just $2\np$ samples of the low-pass filtered version of $f\left(x\right)$,
for an appropriately chosen convolution kernel. After some algebraic
manipulations, the algorithm amounts to recovering \eqref{eq:sum-deltas}
from its Fourier samples 
\begin{equation}
\widehat{f}\left(k\right)=\frac{1}{2\pi}\int_{-\pi}^{\pi}\fun\left(t\right)\ee^{-\imath kt}\dd t=\sum_{j=1}^{\np}\nn{\jc}_{j}\ee^{-\imath\nd_{j}k},\label{eq:fourier-coeffs-spiketrain-1}
\end{equation}
which leads precisely to the \emph{powersum fitting problem \eqref{eq:basic-prony}
}with nodes on the unit circle, i.e. $\left|z_{j}\right|=1$\emph{.}

The function \eqref{eq:sum-deltas} is a special case of what the
authors of \cite{vetterli2002ssf} called a signal with \emph{finite
rate of innovation (FRI)}. Informally, such signals can be described
as having a finite number of degrees of freedom per unit of time.
Many of the FRI sampling architectures proposed since the appearance
of the paper \cite{vetterli2002ssf}, ultimately reduce to the powersum
fitting problem \cite{kusuma2009accuracy}. The underlying idea is
that for perfect reconstruction, it is sufficient to sample such signals
\emph{at their rate of innovation}, and not at their Nyquist rate.
Many types of signals have been shown to be perfectly reconstructable
by FRI techniques, in particular nonuniform splines and piecewise
polynomials \cite{dragotti2007sma}. In the latter case, the following
generalization of the model \eqref{eq:sum-deltas} is considered:
\begin{align}
f(x) & =\sum_{j=1}^{\np}\sum_{\ell=0}^{\ell_{j}-1}\jc_{\ell,j}\der{\delta}{\ell}(x-\nd_{j}),\quad\jc_{\ell,j}\in\reals,\;\nd_{j}\in\reals.\label{eq:gen-delta-fun}
\end{align}
In this case, \eqref{eq:fourier-coeffs-spiketrain-1} becomes, after
a change of variables, the following \emph{polynomial Prony system}
\begin{equation}
\meas=\sum_{j=1}^{\np}z_{j}^{k}\sum_{\ell=0}^{\ell_{j}-1}\nn{\jc}_{\ell,j}k^{\ell},\qquad\nn{\jc}_{\ell,j}\in\complexfield,\;\left|z_{j}\right|=1.\label{eq:polynomial-prony}
\end{equation}

Yet another generalization of \eqref{eq:basic-prony} is the ``confluent
Prony'' system 
\begin{equation}
\meas=\sum_{j=1}^{\np}\sum_{\ell=0}^{\ell_{j}-1}\jc_{\ell,j}\ff k\ell\,\jp_{j}^{k-\ell},\qquad\jp_{j}\in\complexfield\setminus\left\{ 0\right\} ,\;\jc_{\ell,j}\in\complexfield,\label{eq:conf-prony}
\end{equation}
where $\ff{k}{\ell}$ is the Pochhammer symbol for the falling factorial
\[
\ff{k}{\ell}\isdef k(k-1)\cdot\dots\cdot(k-\ell+1).
\]
It appears for example in the problem of reconstructing quadrature
domains from their moments \cite{golub2000snm,gustafsson2000rpd}.

In the remainder of this paper we call \eqref{eq:basic-prony}, \eqref{eq:polynomial-prony}
and \eqref{eq:conf-prony} by the name \emph{Prony-type systems}.
For convenience, we put all the formulas together into \prettyref{tbl:prony-systems}. 

\begin{table}
\centering{}%
\begin{tabular}{|c|c|c|}
\hline 
Type & Formula & Assumptions\tabularnewline
\hline 
Basic  & %
\begin{minipage}[t]{0.4\columnwidth}%
\[
\meas=\sum_{j=1}^{\np}\jc_{j}z_{j}^{k}
\]
\end{minipage} & $\jc_{j},z_{j}\in\complexfield$\tabularnewline
\hline 
Confluent & %
\begin{minipage}[t]{0.4\columnwidth}%
\[
\meas=\sum_{j=1}^{\np}\sum_{\ell=0}^{\ell_{j}-1}\jc_{\ell,j}\ff k\ell z_{j}^{k-i}
\]
\end{minipage} & $\jc_{\ell,j},z_{j}\in\complexfield$\tabularnewline
\hline 
Polynomial & %
\begin{minipage}[t]{0.4\columnwidth}%
\[
\meas=\sum_{j=1}^{\np}z_{j}^{k}\sum_{\ell=0}^{\ell_{j}-1}\jc_{\ell,j}k^{\ell}
\]
\end{minipage} & $\jc_{\ell,j}\in\complexfield,\;\left|z_{j}\right|=1$\tabularnewline
\hline 
\end{tabular}\caption{Prony type systems}
\label{tbl:prony-systems}
\end{table}

In this work (and also in the literature) the unknowns $\left\{ z_{j}\right\} $
(or the corresponding angles $\nd_{j}=\pm\arg z_{j}$) are frequently
called ``poles'', ``nodes'' or ``jumps'', while the linear coefficients
$\left\{ \jc_{\ell,j}\right\} $ are called ``magnitudes''.

We denote the number of unknown coefficients $\jc_{\ell,j}$ by $\ncoeffs\isdef\sum_{i=1}^{\np}\ell_{i}$,
and the overall number of unknown parameters by $\nparams\isdef\ncoeffs+\np$. 

For more details on the algebraic sampling approach we refer to \cite{sig_ack,sarigPhD,vetterli2002ssf}.

An important problem for the applicability of the algebraic reconstruction
techniques is one of stability.
\begin{problem}[Robust Prony solution]
\label{prob:prony-stability-main}How robustly can the parameters
$\left\{ \jc_{\ell,j},z_{j}\right\} $ be recovered from the noisy
data $\left\{ \nn{\meas}=\meas+\er\right\} _{k=0}^{\sc}$? 
\end{problem}
Even more interesting and widely open problem is the following special
case of \prettyref{prob:prony-stability-main}.
\begin{problem}[Problem of superresolution]
\label{prob:superresolution}How robustly can two closely spaced
poles $\left\{ z_{1},z_{2}\right\} $ be recovered from the noisy
data $\left\{ \nn{\meas}=\meas+\er\right\} _{k=0}^{\sc}$?
\end{problem}
Stable solution of Prony-type systems turns out to be a difficult
problem, and in recent years many algorithms have been devised for
this task (e.g. \cite{badeau2006high,badeau2008cramer,badeau2008performance,sampta13Demanet,Holmstrom200231,kahn1992cps,osborne1975some,peter2011nonlinear,potts2010parameter,rao1992mbp,stoica1989music,vanblaricum1978pas}).
The basic idea underlying most of them is some kind of separation
of variables -- it turns out that the nonlinear parameters $\left\{ z_{j}\right\} $
can be found independently of the linear ones $\left\{ \jc_{\ell,j}\right\} $
by a kind of ``elimination''.

Following \cite{rao1992mbp,stoica2005spectral}, \cite{kusuma2009accuracy}
and \cite{sampta13Demanet}, we roughly divide the various methods
into several groups as follows. Note that hybrid approaches such as
\cite{sampta13Demanet} also exist.
\begin{enumerate}
\item \emph{Prony-like methods }(polynomial root finding/annihilating filter,
Pisarenko's method, Approximate Prony method \cite{potts2010parameter}).
These are based on the original Prony's method \cite{prony1795essai}
of constructing a Hankel matrix $H$ from the samples $\left\{ \meas\right\} _{k=0}^{2\ncoeffs-1}$,
finding a vector $\vec v$ in the nullspace of $H$ and constructing
a polynomial from the entries of $\vec v$ whose roots are the unknown
$\left\{ z_{j}\right\} _{j=1}^{\np}$ with the corresponding multiplicites
$\left\{ \ell_{j}\right\} _{j=1}^{\np}.$ These methods are considered
to be numerically unstable. 
\item \emph{Subspace-based methods} (matrix pencils, MUSIC, ESPRIT and generalized
ESPRIT). These methods utilize the special structure of the signal
subspace (manifested in the so-called \emph{rotational invariance
property}), and in general have superior performance for large number
of samples $\sc\gg1$ and Gaussian noise.
\item \emph{Least squares based methods} (Nonlinear Least Squares, Total
Least Squares, Separable Nonlinear Least Squares). Given a good initial
approximation, these methods perform well in many cases, see \cite{golub2003separable}.
\item \emph{Algebraic methods }(Cornell's method \cite{cornell1962method}
and its generalizations, Eckhoff's elimination method \cite{eckhoff1995arf}).
These methods require solving extremely nontrivial nonlinear equations,
explicit formulas are feasible only for small number of nodes, and
they tend to be numerically unstable \cite{kusuma2009accuracy}.
\item \emph{$\ell_{1}$-/total variation minimization} \cite{candes2012towards}.
This recent approach uses ideas from compressed sensing for reconstructing
signals of the form \eqref{eq:sum-deltas} from the first $\sc$ noisy
Fourier samples in \eqref{eq:basic-prony}. Under an explicit node
separation assumption of the form
\[
\Delta\isdef\min_{i\neq j}\left|\jp_{i}-\jp_{j}\right|\geqslant\frac{2}{\sc},
\]
a stable recovery is possible, while the $\ell_{1}$-norm of the error
satisfies
\[
\|\nn{\fun}-\fun\|_{1}\lessapprox\sc^{-2}\err
\]
where $\err$ is the $L_{2}$-error in the input.
\end{enumerate}
Turning to \emph{best possible} \emph{performance of any method whatsoever},
we are aware of two general results -- the Cramer-Rao bounds (CRB)
for the Polynomial Amplitude Complex Exponential (PACE) model \cite{badeau2008cramer},
and Donoho's lower bounds for recovery of sparse measures \cite{donoho1992superresolution}.
The former approach gives fairly elaborate estimates, which are unfortunately
not directly applicable to many problems of algebraic sampling where
no assumptions can be made on the noise except an absolute bound on
its magnitude. The latter bounds are of different kind, however they
are not entirely satisfactory either, since e.g. they provide only
$L_{2}$-norm estimates, while in many applications a bound for $\left|\Delta z_{j}\right|$
is required. We provide more details on these two results in \prettyref{sec:relation-known-bounds}
below.

In \cite{batenkov2011accuracy} we considered the problem of estimating
the best possible accuracy of solving the confluent Prony system \eqref{eq:conf-prony}
from the noisy measurements $\left\{ \nn{\meas}\right\} _{k\in S}$
for the index subset $S=\left\{ 0,1,\dots,\nparams-1\right\} $. The
assumption that the number of equations equals the number of unknowns
is not unreasonable in applications, and furthermore there exist indications
such as the work \cite{peter2011nonlinear}, that increasing the number
of measurements might actually result in deterioration in stability
of solution. Under this assumption, we defined the \emph{local stability}
as the Lipschitz constant of the ``inverse Prony map'', and estimated
this constant at each point in the measurement space $\complexfield^{\nparams}$
where the inverse is defined. We have shown that in this case, if
the noise is bounded in $\ell_{\infty}$ norm by $\err$, then the
local accuracy is bounded as follows:
\begin{align*}
\left|\Delta\jc_{\ell,j}\right| & \leqslant C_{1}\varepsilon\biggl(1+\frac{|\jc_{\ell-1,j}|}{|\jc_{\ell_{j}-1,j}|}\biggr)\qquad0\leq\ell\leq\ell_{j}-1,\; j=1,\dots,\np;\\
\left|\Delta z_{j}\right| & \leqslant C_{1}\err,\qquad j=1,\dots,\np;
\end{align*}
where $C_{1}$ is the maximal row sum norm of the inverse confluent
Vandermonde matrix defined on the nodes $\left\{ z_{1},\dots,z_{\np}\right\} $
with corresponding multiplicities $\left\{ \ell_{1}+1,\dots,\ell_{\np}+1\right\} $.
In fact (see also \prettyref{rem:non-decimated-refinement} below),
in \prettyref{sub:vandermonde} below we show that this constant can
be essentially bounded by $\delta^{-\nparams}$ where $\delta$ is
the node separation
\[
\delta\isdef\min_{i<j}\left|\jp_{i}-\jp_{j}\right|.
\]

The ``Prony map'' method certainly cannot be applied in the case
of oversampling, i.e. when taking $S=\left\{ 0,1,\dots,\sc-1\right\} $
for some large $\sc\gg\nparams$ and solving the resulting system
in some least-squares sense. While oversampling is certainly justified
in the case of noise with a known statistical distribution, it is
not a-priori clear that it would provide any increase in robustness
(and as we pointed out, there are indications to the contrary). That
said, it is natural to assume that one can somehow ``utilize'' the
additional information and obtain a better accuracy of reconstruction,
while staying with small number of measurements. 

Our main goal in this paper is to show that such a utilization is
indeed possible. In particular, we extend the analysis of \cite{batenkov2011accuracy}
to Prony type systems on evenly spaced sampling sets with starting
index $\init$ and step size $\df$:
\begin{equation}
S_{\init,\df}=\left\{ \init,\init+\df,\init+2\df,\dots,\init+\left(\nparams-1\right)\df\right\} .\label{eq:sampling-set-def}
\end{equation}

Such ``decimation'' turns out to retain the essential structure
of the problem, while reducing the Lipschitz constant of the inverse
Prony map. In particular, denoting
\[
\delta_{\df}\isdef\min_{i<j}\left|\jp_{i}^{\df}-\jp_{j}^{\df}\right|,
\]
the error amplification is shown to satisfy (see Theorems \ref{thm:polynomial-stability}
and \ref{thm:confluent-stability})
\begin{align*}
\left|\Delta\jp_{j}\right| & \lessapprox\df^{-\ell_{j}}\delta_{\df}^{-\nparams}\err\\
\left|\Delta\jc_{\ell,j}\right| & \lessapprox\df^{-\ell}\delta_{\df}^{-\nparams+\ell-\ell_{j}}\err.
\end{align*}
Consequently, decimation provides an improvement in accuracy of the
order $\df^{\ell_{j}}$ (see \prettyref{cor:decimation-improvement-acc}
for the precise statement). In qualitative terms, these bounds are
immediately seen to be very similar to the CRB bounds for the PACE
model (see \prettyref{sub:crb-pace}). Numerical experiments, presented
in \prettyref{sec:numerical}, suggest that indeed decimation leads
to improvement in performance of algorithms for solving Prony type
systems.

Furthermore, for closely spaced nodes we have $\delta_{\df}\geqslant\frac{\df\delta}{2}$
for moderate values of $\df$ (see \prettyref{lem:prony-superresolution-deltap}),
and therefore by \prettyref{cor:prony-superresolution-final} in this
case decimation provides improvement in accuracy by the overall factor
of $\df^{\nparams+\ell_{j}}$. This effect can be considered as a
type of ``superresolution''. In fact, we show that the Prony stability
bounds in this case are asymptotically of the same order as Donoho's
bounds -- see \prettyref{sub:donoho-lower-bounds} for details.

A method very similar to decimation, called ``subspace shifting'',
or interleaving, was proposed by Maravic \& Vetterli in \cite{maravic2005sar}
in the context of FRI sampling in the presense of noise. Their idea
was to interleave the rows of the Hankel matrix used in subspace estimation
methods, effectively increasing the separation of closely spaced nodes.
They confirmed this idea with numerical experiments. The results of
this paper can be informally considered as another justification of
their approach.

The system \eqref{eq:polynomial-prony} is of central importance in
Eckhoff's approach to overcoming the Gibbs phenomenon, and has been
a subject of considerable interest to us in this context. In particular,
K.Eckhoff conjectured that the discontinuity locations of a piecewise-smooth
function, having $\ord$ continuous derivatives between the jumps,
can be reconstructed from its first $\sc$ Fourier coefficients with
accuracy $\sim\sc^{-\ord-2}$, by solving a particular instance of
the \emph{noisy }system \eqref{eq:polynomial-prony} \emph{with sufficiently
high accuracy}. The main problem which remained unsolved was: is this
high accuracy indeed achievable, given the assumptions on the noise?
We have recently provided a solution to this problem in \cite{batFullFourier,batyomAlgFourier}.
In \prettyref{sec:fourier} we explain how those results can be reinterpreted
in the framework of stability of Prony type systems, and in particular
the decimation technique.

\section{\label{sec:decimation-optimal-recovery}Decimation and optimal recovery}

\subsection{Definitions}

Now we consider \prettyref{prob:prony-stability-main} for sampling
sets $S_{\init,\df}$ as in \eqref{eq:sampling-set-def} \vpageref{eq:sampling-set-def}.
\begin{defn}
A Prony-type system is called \emph{decimated} if it is considered
with $S_{\init,\df}$ where $\df>1$, and \emph{non-decimated} if
$\df=1$.
\end{defn}
\begin{minipage}[t]{1\columnwidth}%
\end{minipage}
\begin{defn}
A Prony-type system is called \emph{shifted} if it is considered with
$S_{\init,\df}$ where $\init>0$, and \emph{non-shifted} if $\init=0$.\end{defn}
\begin{rem}
\label{rem:non-shifted-equiv-pow}Consider the system \eqref{eq:polynomial-prony}
in the case $\init=0$ and $\df\geqslant1$. It is easy to see that
by making the change of variables
\begin{eqnarray*}
b_{i,j} & = & \jc_{i,j}\df^{i},\\
w_{j} & = & \jp_{j}^{\df},\\
n_{k} & = & \meas[k\df],
\end{eqnarray*}
we arrive at the non-decimated system
\[
n_{k}=\sum_{j=1}^{\np}w_{j}^{k}\sum_{i=0}^{\ell_{j}-1}b_{i,j}k^{i},\quad k=0,1,2,\dots.
\]

\end{rem}
For studying stable recovery, we introduced in \cite{batenkov2011accuracy}
the following framework. 
\begin{defn}
\label{def:accloc}Let ${\cal P}:\complexfield^{\nparams}\to\complexfield^{\nparams}$
be some differentiable mapping, and let $\vec x\in\complexfield^{\nparams}$
be a regular point of ${\cal P}$. Assume that $\varepsilon$ is small
enough so that that the inverse function $\ivm={\cal P}^{-1}$ exists
in $\varepsilon$-neighborhood of $\exams={\cal P}\left(\src\right)$.
For every $1\leqslant r\leqslant\nparams$, let $\left[\vec v\right]_{r}$
denote the $r$-th component of the vector $\vec v\in\complexfield^{\nparams}$.
The best possible \emph{local point-wise accuracy }of inverting ${\cal P}$
with each noise component bounded above by $\varepsilon$ at the point
$\src$ with respect to the component $r$ is
\[
\laccr^{\left({\cal P}\right)}\left(\src,\varepsilon,r\right)\isdef\sup_{\noims\in B\left(\exams,\varepsilon\right)}\left|\left[\jac_{\ivm}(\exams)\left(\noims-\exams\right)\right]_{r}\right|
\]
where $\jac_{\ivm}\left(\exams\right)$ is the Jacobian of $\ivm$
at the point $\exams$.
\end{defn}
\begin{minipage}[t]{1\columnwidth}%
\end{minipage}
\begin{defn}
\label{def:decimated-prony-mappings}The measurement mapping $\ppm:\complexfield^{\nparams}\to\complexfield^{\nparams}$
(resp. $\cpm$) is defined by the Prony equations \eqref{eq:polynomial-prony}
(resp. \eqref{eq:conf-prony}) on $S_{\init,\df}$:
\begin{eqnarray*}
\ppm\left(\left\{ \jp_{j}\right\} ,\left\{ \jc_{i,j}\right\} \right) & = & \left(\meas[\init],\meas[\init+\df],\dots,\meas[\init+\left(\nparams-1\right)\df]\right),\qquad\meas=\sum_{j=1}^{\np}\jp_{j}^{k}\sum_{i=0}^{\ell_{j}-1}\jc_{i,j}k^{i};\\
\cpm\left(\left\{ \jp_{j}\right\} ,\left\{ \jc_{i,j}\right\} \right) & = & \left(\meas[\init],\meas[\init+\df],\dots,\meas[\init+\left(\nparams-1\right)\df]\right),\qquad\meas=\sum_{j=1}^{\np}\sum_{i=0}^{\ell_{j}-1}\jc_{i,j}\ff ki\jp_{j}^{k-i}.
\end{eqnarray*}

\end{defn}

\subsection{Main results}

By factorizing the Jacobians of $\ppm$and $\cpm$ we get the following
results. The  proofs are rather technical and they are presented in
\prettyref{sec:proofs-stability}.
\begin{lem}
\label{lem:unique-solvability}The point $\src=\left(\{\jc_{ij}\},\{\jp_{i}\}\right)\in\complexfield^{\nparams}$
is a regular point of $\ppm$ (resp. $\cpm$) if and only if
\begin{enumerate}
\item $\jp_{j}^{\df}\neq\jp_{i}^{\df}$ for $i\neq j$, and
\item $\jc_{\ell_{j}-1,j}\neq0$ for all $j=1,\dots,\np$.
\end{enumerate}
\end{lem}
\begin{proof}
Use \prettyref{prop:poly-jac-fact} and \prettyref{prop:conf-jac-fact},
as well as \prettyref{prop:cvand-nondegeneracy}.\end{proof}
\begin{thm}
\label{thm:polynomial-stability}Let $\src=\left(\{\jc_{ij}\},\{\jp_{i}\}\right)\in\complexfield^{\nparams}$
be a regular point of $\ppm$. Denote the minimal $\df$-separation
by 
\begin{equation}
\delta_{\df}\isdef\min_{i\neq j}\left|\jp_{i}^{\df}-\jp_{j}^{\df}\right|>0.\label{eq:min-separation-condition-df}
\end{equation}
Then
\begin{eqnarray}
\placcr\left(\src,\varepsilon,\jc_{\ell,j}\right) & = & C_{1}\left(\ell,\ell_{j}\right)\left(\frac{2}{\delta_{\df}}\right)^{\nparams}\left(\frac{1}{2}+\frac{\nparams}{\delta_{\df}}\right)^{\ell_{j}-\ell}\times\label{eq:poly-stability-ampl}\\
 &  & \times\left(1+\frac{\left|\jc_{\ell-1,j}\right|}{\left|\jc_{\ell_{j}-1,j}\right|}\right)\frac{\max\left\{ 1,\init^{\ell_{j}-\ell}\right\} }{\df^{\ell}}\err,\nonumber \\
\placcr\left(\src,\varepsilon,\jp_{j}\right) & = & \frac{2}{\ell_{j}!}\left(\frac{2}{\delta_{\df}}\right)^{\nparams}\frac{1}{\left|\jc_{\ell_{j}-1,j}\right|\df^{\ell_{j}}}\err,\label{eq:poly-stability-nodes}
\end{eqnarray}
where $C_{1}\left(\ell,\ell_{j}\right)$ is an explicit constant defined
in \eqref{eq:cij-l-lj-def} \vpageref{eq:cij-l-lj-def} below.
\end{thm}
\begin{minipage}[t]{1\columnwidth}%
\end{minipage}
\begin{thm}
\label{thm:confluent-stability}Let $\src=\left(\{\jc_{ij}\},\{\jp_{i}\}\right)\in\complexfield^{\nparams}$
be a regular point of $\cpm$. Assume that the nodes $\left\{ \jp_{j}\right\} _{j=1}^{\np}$
of the confluent Prony system \eqref{eq:conf-prony} \vpageref{eq:conf-prony}
satisfy the condition \eqref{eq:min-separation-condition-df} \vpageref{eq:min-separation-condition-df},
and also that $0<\left|\jp_{j}\right|\leqslant1$ for $j=1,\dots,\np$.

Then
\begin{eqnarray*}
\claccr\left(\src,\varepsilon,\jc_{\ell,j}\right) & \leqslant & C_{2}\left(\ell,\ell_{j}\right)\left(\frac{2}{\delta_{\df}}\right)^{\nparams}\left(\frac{1}{2}+\frac{\nparams}{\delta_{\df}}\right)^{\ell_{j}-\ell}\left|\jp_{j}\right|^{\ell-\init-\df\ell_{j}}\times\\
 &  & \times\left(1+\frac{\left|\jc_{\ell-1,j}\right|}{\left|\jc_{\ell_{j}-1,j}\right|}\right)\frac{\max\left\{ 1,\init^{\ell_{j}-\ell}\right\} }{\df^{\ell}}\err,\\
\claccr\left(\src,\varepsilon,\jp_{j}\right) & \leqslant & \frac{2}{\ell_{j}!}\left(\frac{2}{\delta_{\df}}\right)^{\nparams}\frac{\left|\jp_{j}\right|^{\ell_{j}-\init-\df\ell_{j}}}{\left|\jc_{\ell_{j}-1,j}\right|\df^{\ell_{j}}}\err,
\end{eqnarray*}
where $C_{2}\left(\ell,\ell_{j}\right)$ is an explicit constant defined
in \eqref{eq:cij-l-lj-def-1} \vpageref{eq:cij-l-lj-def-1} below.

\end{thm}
\begin{rem}
Note that the bounds of \prettyref{thm:polynomial-stability} and
\prettyref{thm:confluent-stability} coincide in the case of $\ell_{1}=\dots=\ell_{\np}=1$
and $\left|\jp_{j}\right|=1$, in which case we just have the original
Prony system \eqref{eq:basic-prony} \vpageref{eq:basic-prony}. 
\end{rem}
\begin{minipage}[t]{1\columnwidth}%
\end{minipage}
\begin{rem}
The noise amplification for recovering the nodes does not depend on
the initial index $\init$, while the accuracy of recovering $\left|\jc_{\ell,j}\right|$
actually deteriorates with increasing $\init$ (if keeping $\df$
constant).
\end{rem}
\begin{minipage}[t]{1\columnwidth}%
\end{minipage}
\begin{rem}
If $\init=0$, then one can make the change of variables described
in \prettyref{rem:non-shifted-equiv-pow} above and derive \prettyref{thm:polynomial-stability}
from the results of \cite{batenkov2011accuracy} on non-shifted and
non-decimated Prony system.
\end{rem}
\begin{minipage}[t]{1\columnwidth}%
\end{minipage}
\begin{rem}
\label{rem:aliasing}If $\df>1$ (non-trivial decimation) we have
a problem of ``aliasing'' or non-uniqueness. In effect, instead
of $\jp_{j}$ one recovers $\jp_{j}^{\df}$. Therefore the decimation
cannot be used on general Prony systems without having a good a-priori
approximation to the nodes (at least with accuracy $\sim o\left(\df^{-1}\right)$).
\end{rem}
\begin{minipage}[t]{1\columnwidth}%
\end{minipage}
\begin{rem}
\label{rem:non-decimated-refinement}Taking \prettyref{thm:confluent-stability}
with $\init=0,\;\df=1$ we obtain a refinement of the main result
of \cite{batenkov2011accuracy}.
\end{rem}

\subsection{Improvement gained by decimation}

We subseqeuently consider the following question: \emph{how much improvement
can one get by decimation?}

To be more specific, we take $\init=0$ (non-shifted system), and
consider the quantity $\laccr^{\left(\fwm[0][\df]P\right)}\left(\src,\varepsilon,\jp_{j}\right)$
as function of the decimation $\df$, for $1\leqslant\df\leqslant\frac{\ns}{\nparams-1}$
(so that we always have $S_{0,\df}\subset\left[0,\ns\right]$). Keeping
$\err$ constant, we have essentially
\[
\laccr^{\left(\fwm[0][\df]P\right)}\left(\src,\varepsilon,\jp_{j}\right)\sim\delta_{\df}^{-\nparams}\df^{-\ell_{j}}.
\]

\begin{defn}
The ``improvement function'' is defined as the ratio between non-decimated
and decimated error amplification, i.e.
\[
\rho\left(\src,\jp_{j},\df\right)\isdef\frac{\laccr^{\left(\fwm[0][1]P\right)}\left(\src,\varepsilon,\jp_{j}\right)}{\laccr^{\left(\fwm[0][\df]P\right)}\left(\src,\varepsilon,\jp_{j}\right)}=\left(\frac{\delta_{\df}}{\delta}\right)^{\nparams}\df^{\ell_{j}}.
\]

\end{defn}
In particular, accuracy is increased whenever $\rho\left(\src,\jp_{j},\df\right)>1$,
or
\begin{equation}
\frac{\df^{-\ell_{j}}}{\delta_{\df}{}^{\nparams}}<\frac{1}{\delta^{\nparams}}.\label{eq:acc-incr-cond}
\end{equation}

In order to keep things simple but still nontrivial, we shall deal
exclusively with the case $\np=2$.
\begin{cor}
\label{cor:decimation-improvement-acc}Let $\src=\left(\{\jc_{ij}\},\{\jp_{i}\}\right)\in\complexfield^{\nparams}$
be a regular point of $\fwm[0][1]P$. Then there exists a constant
$C_{3}=C_{3}\left(\src\right)$ such that for any $M\in\naturals$
there exists $\df=\df\left(M,\src\right)>M$ for which
\[
\rho\left(\src,\jp_{j},\df\right)\geqslant C_{3}\df^{\ell_{j}}.
\]
\end{cor}
\begin{proof}
Without loss of generality, let the jumps be $\jp_{1}=1,\ \jp_{2}=\ee^{-\imath\nd}$.
We have 
\[
\rho\left(\vec x,\jp_{j},\df\right)=\left(\frac{\delta_{\df}}{\delta}\right)^{\nparams}\df^{\ell_{j}}=\left|\frac{1-\ee^{-\imath\df\nd}}{1-\ee^{-\imath\nd}}\right|^{\nparams}\df^{\ord}=\left(\frac{\sin\frac{\df\nd}{2}}{\sin\frac{\nd}{2}}\right)^{\nparams}\df^{\ell_{j}}.
\]

Let $\alpha\isdef\frac{\nd}{2}$, and let $n\in\naturals$ be arbitrary.
Clearly, we have
\begin{equation}
\exists\df_{0}=\df_{0}\left(n,\nd\right)>n:\quad\left|\sin\df_{0}\alpha\right|>\frac{1}{2}.\label{eq:temp-claim}
\end{equation}
Thus, for this $\df_{0}$ we have
\[
\rho\left(x,\jp_{j},\df_{0}\right)\geqslant\frac{\df_{0}^{\ell_{j}}}{\left(2\sin\frac{\nd}{2}\right)^{\nparams}}
\]
which proves the claim.
\end{proof}

\subsection{\label{sub:prony-superres}Superresolution}

We return to \prettyref{prob:superresolution} \vpageref{prob:superresolution}.
In this section we estimate the effect of decimation on  closely spaced
nodes. For simplicity, let us consider the case of polynomial Prony
system \eqref{eq:polynomial-prony} \vpageref{eq:polynomial-prony}
with only two nodes. Below we show that for moderate values of the
decimation parameter $\df$, accuracy is improved by an additional
factor $\df^{\nparams}$.
\begin{lem}
\label{lem:prony-superresolution-deltap}Consider the case of polynomial
Prony system \eqref{eq:polynomial-prony} with two nodes $z_{1}=\ee^{\imath\nd_{1}}$
and $\jp_{2}=\ee^{\imath\nd_{2}}$, so that $\nd_{2}=\nd_{1}+\delta$
and $0<\delta\ll1$. Fix some $0<r_{0}<2\pi.$ For any $\df$ satisfying
\[
\df\delta<r_{0}\iff\df<\left\lfloor \frac{r_{0}}{\delta}\right\rfloor 
\]
we have
\[
\left|\delta_{\df}\right|>\alpha\left(r_{0}\right)\df\delta,
\]
where
\[
\alpha\left(r\right)\isdef\frac{\sqrt{2\left(1-\cos r\right)}}{r}.
\]
\end{lem}
\begin{proof}
We have
\begin{eqnarray*}
\left|\delta_{\df}\right|^{2} & = & \left|\ee^{\imath\nd_{2}\df}-\ee^{\imath\nd_{1}\df}\right|^{2}=\left|1-\ee^{\imath\delta\df}\right|^{2}=2\left(1-\cos\left(\df\delta\right)\right)\\
 & = & \frac{2\left(1-\cos\left(\df\delta\right)\right)}{\left(\df\delta\right)^{2}}\left(\df\delta\right)^{2}\\
 & = & \left[\alpha\left(\df\delta\right)\right]^{2}\left(\df\delta\right)^{2}.
\end{eqnarray*}
Since the function $\alpha\left(r\right)$ is monotonically decreasing
in $0<r<2\pi$ (see \prettyref{fig:alpha-fun}), and since $\df\delta<r_{0}$,
we have immediately $\alpha\left(\df\delta\right)>\alpha\left(r_{0}\right)$,
which completes the proof.\end{proof}
\begin{cor}
\label{cor:prony-superresolution-final}Under the conditions of \prettyref{lem:prony-superresolution-deltap},
\[
\rho\left(\src,\jp_{j},\df\right)>\alpha\left(r_{0}\right)^{\nparams}\df^{\nparams+\ell_{j}},\qquad j=1,2.
\]

\end{cor}
\begin{figure}[h]
\begin{centering}
\includegraphics{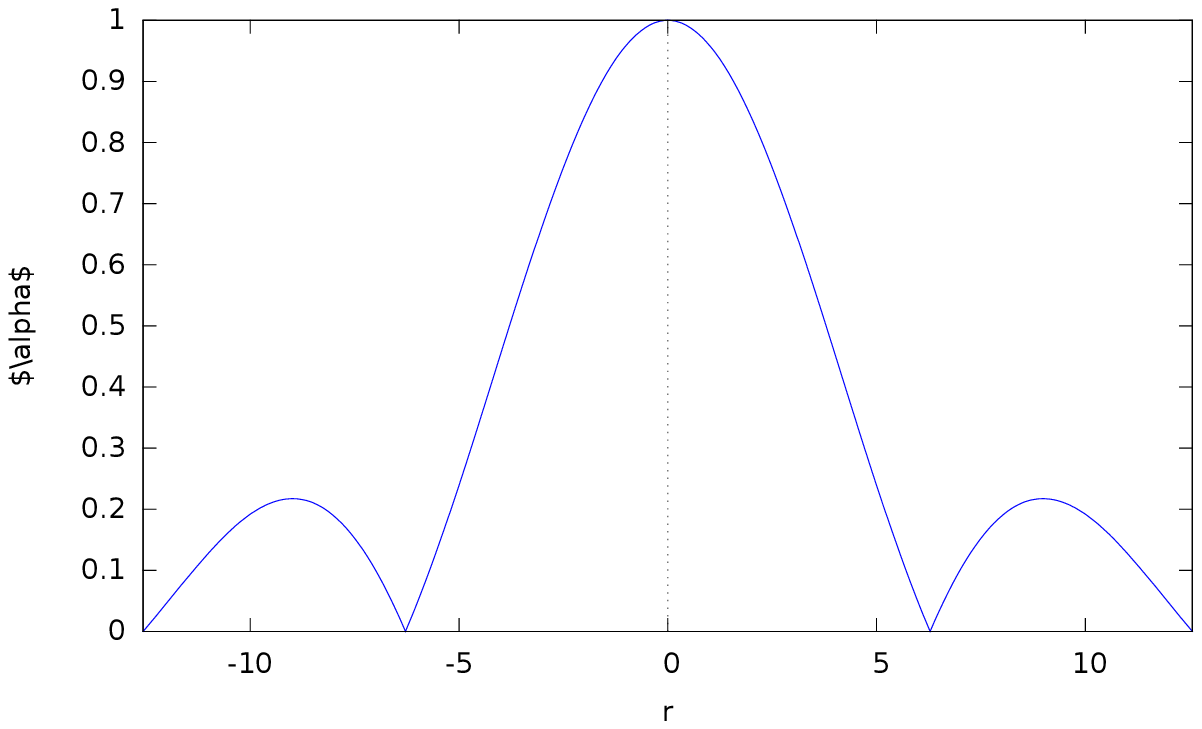}
\par\end{centering}

\caption{The function $\alpha\left(r\right).$}
\label{fig:alpha-fun}
\end{figure}

\newcommandx\fc[1][usedefault, addprefix=\global, 1=k]{\ensuremath{c_{#1}}}
\newcommandx\frsum[2][usedefault, addprefix=\global, 1=\fun, 2=\sc]{\mathfrak{F}_{#2}\left(#1\right)}
\global\long\def\smooth{\ensuremath{\Psi}}
\global\long\def\sing{\ensuremath{\Phi}}
\global\long\def\sc{\ensuremath{M}}
\global\long\def\jcc{\alpha}
\global\long\def\scc{N}
\global\long\def\w{\ensuremath{\omega}}

\section{\label{sec:proofs-stability}Proofs of main results}

This section contains the proofs of the theorems stated in \prettyref{sec:decimation-optimal-recovery}.
Most of the technical propositions regarding matrix factorizations
are straightforward, using nothing more than some elementary algebra
and binomial identities. Therefore, we have omitted most of these
calculations, confident that the reader would reproduce them without
any difficulty.

\subsection{Common definitions}

We start by defining the matrices which will be used throughout the
subsequent calculations.

\begin{minipage}[t]{1\columnwidth}%
\end{minipage}
\begin{defn}
Denote by $\cfmat_{j}$ the following $\ell_{j}\times\ell_{j}$ block:
\[
\cfmat_{j}\isdef\begin{bmatrix}\jc_{0,j} & \jc_{1,j} & \cdots & \cdots & \jc_{\ell_{j}-1,j}\\
\jc_{1,j} &  &  & {\ell_{j}-1 \choose \ell_{j}-2}\jc_{\ell_{j}-1,j} & 0\\
\cdots &  &  & \cdots & 0\\
 & {\ell_{j}-1 \choose 2}\jc_{\ell_{j}-1,j} & 0 & \cdots & 0\\
\jc_{\ell_{j}-1,j} & 0 & \cdots & \cdots & 0
\end{bmatrix}.
\]

\end{defn}
\begin{minipage}[t]{1\columnwidth}%
\end{minipage}
\begin{defn}
Let $\mult[r]$ denote the $r\times r$ square matrix with entries%
\footnote{Note that the matrices $\mult[r]$ are unipotent, i.e. all their eigenvalues
are 1.%
}:
\[
\left(\mult[r]\right)_{m,n}=(-\init)^{n-m}\binom{n-1}{n-m}
\]
\end{defn}
\begin{example}
For $r=5$ we have
\[
\mult[5]=\left(\begin{array}{ccccc}
1 & -\init & \init^{2} & -\init^{3} & \init^{4}\\
0 & 1 & -2\init & 3\init^{2} & -4\init^{3}\\
0 & 0 & 1 & -3\init & 6\init^{2}\\
0 & 0 & 0 & 1 & -4\init\\
0 & 0 & 0 & 0 & 1
\end{array}\right)
\]

\end{example}
\begin{minipage}[t]{1\columnwidth}%
\end{minipage}
\begin{defn}
For every $x\in\complexfield\setminus\left\{ 0\right\} $ and $c\in\naturals$
let $\powermat xc$ denote the $c\times c$ matrix
\[
\powermat xc\isdef\diag\left\{ 1,x,x^{2},\dots,x^{c-1}\right\} .
\]

\end{defn}
Obviously,
\[
\powermat xc^{-1}=\powermat{x^{-1}}c.
\]

In addition we need the following auxiliary matrices.
\begin{defn}
For every $j=1,\dotsc,\np$ let us denote by $\pcoeffbl$ and $\ccoeffbl$
the following $(\ell_{j}+1)\times(\ell_{j}+1)$ blocks
\begin{eqnarray}
\pcoeffbl & \isdef & \left[\begin{array}{cccc}
1 & 0 & 0 & 0\\
0 & 1 & 0 & \frac{\jc_{0,j}}{\jp_{j}}\\
\vdots & \vdots & \ddots & \vdots\\
0 & 0 & 0 & \frac{\jc_{\ell_{j}-1,j}}{\jp_{j}}
\end{array}\right],\label{eq:coeffbl-def}\\
\ccoeffbl & \isdef & \left[\begin{array}{cccc}
1 & 0 & 0 & 0\\
0 & 1 & 0 & \jc_{0,j}\\
\vdots & \vdots & \ddots & \vdots\\
0 & 0 & 0 & \jc_{\ell_{j}-1,j}
\end{array}\right].\label{eq:ccoeffbl-def}
\end{eqnarray}

\end{defn}
\begin{minipage}[t]{1\columnwidth}%
\end{minipage}

Direct calculation gives
\begin{align}
\pcoeffbl^{-1} & =\begin{bmatrix}1 & 0 & 0 & \dots & 0\\
0 & 1 & 0 & \dots & -\frac{\jc_{0,j}}{\jc_{\ell_{j}-1,j}}\\
0 & 0 & 1 & \dots & -\frac{\jc_{1,j}}{\jc_{\ell_{j}-1,j}}\\
 &  &  &  & \vdots\\
0 & 0 & 0 & \dots & +\frac{\jp_{j}}{\jc_{\ell_{j}-1,j}}
\end{bmatrix},\label{eq:pcoeffbl-inv}\\
\ccoeffbl[j]^{-1} & =\begin{bmatrix}1 & 0 & 0 & \dots & 0\\
0 & 1 & 0 & \dots & -\frac{\jc_{0,j}}{\jc_{\ell_{j}-1,j}}\\
0 & 0 & 1 & \dots & -\frac{\jc_{1,j}}{\jc_{\ell_{j}-1,j}}\\
 &  &  &  & \vdots\\
0 & 0 & 0 & \dots & +\frac{1}{\jc_{\ell_{j}-1,j}}
\end{bmatrix}.\label{eq:ccoefbl-inv}
\end{align}

The Stirling numbers of the first and second kind are denoted by $\stf nk$
and $\sts nk$, respectively \cite[Section 24.1]{abramowitz1965handbook}.

\begin{minipage}[t]{1\columnwidth}%
\end{minipage}
\begin{defn}
Let $\stirlmat m$ denote the $m\times m$ upper triangular matrix
\[
\stirlmat m\isdef\left[\begin{array}{ccccc}
\sts 00 & \sts 10 & \sts 20 & \dots & \sts{m-1}0\\
0 & \sts 11 & \sts 21 & \dots & \sts{m-1}1\\
\vdots & \vdots & \ddots &  & \vdots\\
0 & 0 & 0 &  & \sts{m-1}{m-1}
\end{array}\right].
\]

\end{defn}
It is well-known that 
\[
\stirlmat m^{-1}=\left[\begin{array}{ccccc}
\stf 00 & \stf 10 & \stf 20 & \dots & \stf{m-1}0\\
0 & \stf 11 & \stf 21 & \dots & \stf{m-1}1\\
\vdots & \vdots & \ddots &  & \vdots\\
0 & 0 & 0 &  & \stf{m-1}{m-1}
\end{array}\right].
\]

\subsection{\label{sub:vandermonde}Generalized Vandermonde matrices}
\begin{defn}
The \emph{shifted decimated Pascal-Vandermonde} matrix is 
\[
\cpvand_{\init,\df}\left(\jp_{1},\ell_{1},\dots,\jp_{\np},\ell_{\np}\right)=\cpvand_{\init,\df}\isdef\begin{bmatrix}\prow 0\left(\jp_{1},\ell_{1}\right) & \prow 0\left(\jp_{2},\ell_{2}\right) & \dots & \prow 0\left(\jp_{\np},\ell_{\np}\right)\\
\prow 1\left(\jp_{1},\ell_{1}\right) & \prow 1\left(\jp_{2},\ell_{2}\right) & \dots & \prow 1\left(\jp_{\np},\ell_{\np}\right)\\
\vdots & \vdots & \vdots & \vdots\\
\prow{\ncoeffs-1}\left(\jp_{1},\ell_{1}\right) & \prow{\ncoeffs-1}\left(\jp_{2},\ell_{2}\right) & \dots & \prow{\ncoeffs-1}\left(\jp_{\np},\ell_{\np}\right)
\end{bmatrix}
\]
where 
\[
\prow k\left(\jp_{j},\ell_{j}\right)\isdef\jp_{j}^{\init+k\df}\left[\begin{array}{ccccc}
1 & \left(\init+k\df\right) & \left(\init+k\df\right)^{2} & \dots & \left(\init+k\df\right)^{\ell_{j}-1}\end{array}\right].
\]
The \emph{non-shifted Pascal-Vandermonde} matrix on the nodes $\left\{ \jp_{1}^{\df},\dots,\jp_{\np}^{\df}\right\} $
is denoted by 
\[
\cpvand_{\df}^{\#}\isdef\cpvand_{0,1}\left(\jp_{1}^{\df},\ell_{1},\dots,\jp_{\np}^{\df},\ell_{\np}\right).
\]

\end{defn}
\begin{minipage}[t]{1\columnwidth}%
\end{minipage}
\begin{defn}
The \emph{shifted decimated confluent Vandermonde} matrix is
\[
\cvand_{\init,\df}\left(\jp_{1},\ell_{1},\dots,\jp_{\np},\ell_{\np}\right)=\cvand_{\init,\df}\isdef\begin{bmatrix}\crow 0\left(\jp_{1},\ell_{1}\right) & \crow 0\left(\jp_{2},\ell_{2}\right) & \dots & \crow 0\left(\jp_{\np},\ell_{\np}\right)\\
\crow 1\left(\jp_{1},\ell_{1}\right) & \crow 1\left(\jp_{2},\ell_{2}\right) & \dots & \crow 1\left(\jp_{\np},\ell_{\np}\right)\\
\vdots & \vdots & \vdots & \vdots\\
\crow{\ncoeffs-1}\left(\jp_{1},\ell_{1}\right) & \crow{\ncoeffs-1}\left(\jp_{2},\ell_{2}\right) & \dots & \crow{\ncoeffs-1}\left(\jp_{\np},\ell_{\np}\right)
\end{bmatrix}
\]
where
\[
\crow k\left(\jp_{j},\ell_{j}\right)\isdef\left[\begin{array}{cccc}
\jp_{j}^{\init+k\df}, & \left(\init+k\df\right)\jp_{j}^{\init+k\df-1}, & \dots & ,\ff{\init+k\df}{\ell_{j}-1}\jp_{j}^{\init+k\df-\ell_{j}+1}\end{array}\right].
\]
The \emph{non-shifted confluent Vandermonde} matrix on the nodes $\left\{ \jp_{1}^{\df},\dots,\jp_{\np}^{\df}\right\} $
is denoted by 
\[
\cvand_{\df}^{\#}\isdef\cvand_{0,1}\left(\jp_{1}^{\df},\ell_{1},\dots,\jp_{\np}^{\df},\ell_{\np}\right).
\]

\end{defn}
The confluent Vandermonde matrix $\cvand_{0,1}$ is well-studied in
numerical analysis due to its central role in polynomial interpolation.

Let us start with the well-known fact about these matrices.
\begin{prop}
\label{prop:cvand-nondegeneracy}The matrix $\cvand_{0,1}\left(\jp_{1},\ell_{1},\dots,\jp_{\np},\ell_{\np}\right)$
is invertible if and only if the nodes $\left\{ \jp_{j}\right\} _{j=1}^{\np}$
are pairwise distinct.
\end{prop}
Of particular interest to us are estimates on the row-wise norm of
$\cvand_{0,1}^{-1}$.
\begin{thm}
\label{thm:vandermonde-inv-bounds}Let $\{x_{1},\dots,x_{n}\}$ be
pairwise distinct complex numbers with $\left|x_{j}\right|\leq1$,
satisfying the separation condition $\left|x_{i}-x_{j}\right|\geq\delta>0$
for $i\neq j$. Further, let $\{\ell_{1},\dots,\ell_{n}\}$ be a vector
of natural numbers such that $\ell_{1}+\ell_{2}+\dots+\ell_{n}=N$.
Denote by $u_{j,k}$ the row with index $\ell_{1}+\dots+\ell_{j-1}+k+1$
of $\left[\cvand_{0,1}\left(x_{1},\ell_{1},\dots,x_{n},\ell_{n}\right)\right]^{-1}$
(for $k=0,1,\dots,\ell_{j}-1$). Then the $\ell_{1}$-norm of $u_{j,k}$
satisfies
\begin{equation}
\|u_{j,k}\|_{1}\leqslant\left(\frac{2}{\delta}\right)^{N}\frac{2}{k!}\left(\frac{1}{2}+\frac{N}{\delta}\right)^{\ell_{j}-1-k}.\label{eq:main-bound}
\end{equation}

\end{thm}
The proof of this theorem (see below) combines original Gautschi's
technique \cite{gautschi1963inverses} and the well-known explicit
expressions for the entries of $\cvand_{0,1}^{-1}$ from \cite{schappelle1972icv},
plus a technical lemma (\prettyref{lem:technical-lemma}).
\begin{defn}
For $j=1,\dots,n$ let
\begin{equation}
h_{j}(x)=\prod_{i\neq j}(x-x_{i})^{-\ell_{i}}.\label{eq:h-def}
\end{equation}
\end{defn}
\begin{lem}[\cite{248559}]
\label{lem:technical-lemma}For any natural $k$, the $k$-th derivative
of $h_{j}$ at $x_{j}$ satisfies
\[
\left|h_{j}^{(k)}\left(x_{j}\right)\right|\leqslant N(N+1)\cdots(N+k-1)\delta^{-N-k}.
\]
\end{lem}
\begin{proof}
We proceed by induction on $k$. For $k=0$ we have immediately $\left|h_{j}\left(x_{j}\right)\right|\leqslant\delta^{-N}$.
Now
\begin{equation}
h'_{j}(x)=h_{j}(x)\sum_{i\ne j}\frac{-\ell_{i}}{x-x_{i}}\label{eq:log-derivative}
\end{equation}
 Now we apply the Leibnitz rule and get
\begin{eqnarray*}
h_{j}^{\left(k\right)}\left(x\right) & = & \left(\frac{h_{j}'}{h_{j}}h_{j}\right)^{\left(k-1\right)}\\
 & = & \sum_{r=0}^{k-1}{k-1 \choose r}h_{j}^{(r)}(x)\left(\frac{h'_{j}}{h_{j}}\right)^{\left(k-1-r\right)}\\
 & = & \sum_{r=0}^{k-1}{k-1 \choose r}h_{j}^{(r)}(x)\sum_{i\ne j}\frac{(-1)^{k-r-1}(k-r-1)!\ell_{i}}{(x-x_{i})^{k-r}},
\end{eqnarray*}
hence
\[
\left|h_{j}^{(k)}(x_{j})\right|\leqslant\sum_{r=0}^{k-1}{k-1 \choose r}\left|h_{j}^{(r)}(x_{j})\right|\sum_{i\ne j}\frac{(k-r-1)!\ell_{i}}{|x_{j}-x_{i}|^{k-r}}.
\]
This implies, together with the induction hypothesis, that
\begin{eqnarray*}
\left|h_{j}^{(k)}(x_{j})\right| & \leqslant & \sum_{r=0}^{k-1}{k-1 \choose r}\frac{N\left(N+1\right)\cdots\left(N+r-1\right)}{\delta^{N+r}}\cdot\frac{(k-r-1)!N}{\delta^{k-r}}\\
 & = & \frac{N}{\delta^{N+k}}\sum_{r=0}^{k-1}\frac{\left(k-1\right)!}{r!}N\left(N+1\right)\cdots\left(N+r-1\right)\\
 & = & \frac{\left(k-1\right)!N}{\delta^{N+k}}\sum_{r=0}^{k-1}{N-1+r \choose r}.
\end{eqnarray*}
By a well-known binomial identity (proof is immediate by induction
and Pascal's rule) we have
\[
\sum_{r=0}^{k-1}{N-1+r \choose r}={N+k-1 \choose k-1}.
\]
Therefore
\[
\left|h_{j}^{\left(k\right)}\left(x_{j}\right)\right|\leqslant\frac{N\left(N+1\right)\cdots\left(N+k-1\right)}{\delta^{N+k}},
\]
as required.
\end{proof}
\begin{minipage}[t]{1\columnwidth}%
\end{minipage}
\begin{proof}[Proof of \prettyref{thm:vandermonde-inv-bounds}]
By using a generalization of the Hermite interpolation formula (\cite{spitzbart1960generalization}),
it is shown in \cite{schappelle1972icv} that the components of the
row $u_{j,k}$ are just the coefficients of the polynomial
\[
\frac{1}{k!}\sum_{t=0}^{\ell_{j}-1-k}\frac{1}{t!}h_{j}^{(t)}(x_{j})(x-x_{j})^{k+t}\prod_{i\neq j}(x-x_{i})^{\ell_{i}}
\]
where $h_{j}\left(x\right)$ is given by \eqref{eq:h-def}.By \cite[Lemma]{gautschi1962iva},
the sum of absolute values of the coefficients of the polynomials
$(x-x_{j})^{k+t}\prod_{i\neq j}(x-x_{i})^{\ell_{i}}$ is at most
\[
(1+|x_{j}|)^{k+t}\prod_{i\neq j}(1+|x_{i}|)^{\ell_{i}}\leqslant2^{N-(\ell_{j}-k-t)}.
\]
Therefore
\begin{eqnarray*}
\|u_{j,k}\|_{1} & \leqslant & \frac{1}{k!}\sum_{t=0}^{\ell_{j}-1-k}\frac{1}{t!}\frac{N(N+1)\cdots(N+t-1)}{{\delta^{N+t}}}2^{N-\ell_{j}+k+t}\\
 & = & \biggl(\frac{2}{\delta}\biggr)^{N}\frac{1}{{2^{\ell_{j}-k}k!}}\sum_{t=0}^{\ell_{j}-1-k}{\ell_{j}-1-k \choose t}\frac{N(N+1)\cdots(N+t-1)}{(\ell_{j}-k-t)\cdots(\ell_{j}-k-2)(\ell_{j}-k-1)}\biggl(\frac{2}{\delta}\biggr)^{t}\\
 & \leqslant & \biggl(\frac{2}{\delta}\biggr)^{N}\frac{1}{{2^{\ell_{j}-k}k!}}\biggl(1+\frac{2N}{\delta}\biggr)^{\ell_{j}-1-k}\\
 & = & \biggl(\frac{2}{\delta}\biggr)^{N}\frac{2}{k!}\biggl(\frac{1}{2}+\frac{N}{\delta}\biggr)^{\ell_{j}-1-k}
\end{eqnarray*}
which completes the proof.\end{proof}
\begin{cor}
\label{cor:rowwise-norm-cvand-2}Assume that $\left|\jp_{j}\right|\leqslant1$,
and $\df\geqslant1$ is chosen such that
\[
\delta_{\df}\isdef\min_{i\neq j}\left|\jp_{i}^{\df}-\jp_{j}^{\df}\right|>0.
\]
Denote by $u_{j,k}$ the row with index $\ell_{1}+1+\dots+\ell_{j-1}+1+k+1$
of $\left\{ V_{\df}^{\#}\left(\jp_{1},\ell_{1}+1,\dots,\jp_{\np},\ell_{\np}+1\right)\right\} ^{-1}$
(for $k=0,1,\dots,\ell_{j}$). Then the $\ell_{1}$-norm of $u_{j,k}$
satisfies
\begin{equation}
\|u_{j,k}\|_{1}\leqslant\left(\frac{2}{\delta_{\df}}\right)^{\nparams}\frac{2}{k!}\left(\frac{1}{2}+\frac{\nparams}{\delta_{\df}}\right)^{\ell_{j}-k}\label{eq:cvand-norm-bound-2}
\end{equation}
where $\nparams=\sum_{j=1}^{\np}\left(\ell_{j}+1\right)=\ncoeffs+\np$.\end{cor}
\begin{prop}
For the Pascal-Vandermonde matrix, we have
\begin{eqnarray}
\cpvand_{\init,\df} & = & \cpvand_{0,\df}\diag\left\{ \jp_{j}^{\init}\mult[\ell_{j}]^{-1}\right\} _{j=1}^{\np}\label{eq:vandermonde-tkm-fac}\\
\cpvand_{0,\df} & = & \cpvand_{\df}^{\#}\diag\left\{ T_{\df,\ell_{j}}\right\} _{j=1}^{\np}.\label{eq:vandermonde-km-vac}
\end{eqnarray}

\end{prop}
\begin{minipage}[t]{1\columnwidth}%
\end{minipage}
\begin{prop}
If $\np=1$ then in addition to the above factorizations we have
\[
\cpvand_{\df}^{\#}\left(\jp,\ell\right)=\powermat{\jp^{\df-1}}{\ell}\cpvand_{0,1}\left(\jp,\ell\right).
\]

\end{prop}
\begin{minipage}[t]{1\columnwidth}%
\end{minipage}

In addition, we have the following identity.
\begin{prop}
The confluent Vandermonde and Pascal-Vandermonde matrices satisfy
\begin{eqnarray*}
\cpvand_{\init,\df}\left(\jp_{1},\ell_{1},\dots,\jp_{\np},\ell_{\np}\right) & = & \cvand_{\init,\df}\left(\jp_{1},\ell_{1},\dots,\jp_{\np},\ell_{\np}\right)\diag\left\{ \powermat{\jp_{j}}{\ell_{j}}\stirlmat{\ell_{j}}\right\} _{j=1}^{\np},
\end{eqnarray*}
and therefore also
\[
\cpvand_{\df}^{\#}=\cvand_{\df}^{\#}\diag\left\{ \powermat{\jp_{j}^{\df}}{\ell_{j}}\stirlmat{\ell_{j}}\right\} _{j=1}^{\np}.
\]

\end{prop}

\subsection{Data matrices}
\begin{defn}
The data matrix $\pdm$ (resp. $\cdm$) for the system \eqref{eq:polynomial-prony}
(resp. \eqref{eq:conf-prony}) is the Hankel matrix
\[
\left[\begin{array}{cccc}
\meas[\init] & \meas[\init+\df] & \dots & \meas[\init+\left(\ncoeffs-1\right)\df]\\
\meas[\init+\df] & \meas[\init+2\df] & \dots & \meas[\init+\ncoeffs\df]\\
\vdots & \vdots & \vdots & \vdots\\
\meas[\init+\left(\ncoeffs-1\right)\df] & \meas[\init+\ncoeffs\df] & \dots & \meas[\init+\left(2\ncoeffs-2\right)\df]
\end{array}\right]
\]
where $\meas$ are given by \eqref{eq:polynomial-prony} (resp. \eqref{eq:conf-prony}).\end{defn}
\begin{prop}
The data matrices are factorized as follows:
\begin{eqnarray*}
\cdm & = & \cvand_{\init,\df}\diag\left\{ \cfmat_{j}\right\} _{j=1}^{\np}\cvand_{0,\df}^{T},\\
\pdm & = & \cpvand_{\init,\df}\diag\left\{ \cfmat_{j}\right\} _{j=1}^{\np}\cpvand_{0,\df}^{T}.
\end{eqnarray*}

\end{prop}

\subsection{Jacobian factorizations}

Recall the definition of the maps $\ppm$ and $\cpm$ given by \prettyref{def:decimated-prony-mappings}.

Again, the following two results can be shown via direct calculations,
therefore we omit the proofs.
\begin{prop}
\label{prop:poly-jac-fact}The Jacobian matrix of the map $\ppm$
can be decomposed as follows:
\begin{eqnarray*}
\jac_{\ppm} & = & \cpvand_{\init,\df}\left(\jp_{1},\ell_{1}+1,\dots,\jp_{\np},\ell_{\np}+1\right)\diag\{\pcoeffbl[j]\}_{j=1}^{\np}\\
 & = & \cpvand_{\df}^{\#}\left(\jp_{1},\ell_{1}+1,\dots,\jp_{\np},\ell_{\np}+1\right)\diag\left\{ \jp_{j}^{t}\powermat{\df}{\ell_{j}+1}\mult[\ell_{j}+1]^{-1}\pcoeffbl[j]\right\} _{j=1}^{\np}\\
 & = & \cvand_{\df}^{\#}\left(\jp_{1},\ell_{1}+1,\dots,\jp_{\np},\ell_{\np}+1\right)\diag\left\{ \jp_{j}^{t}\powermat{\jp_{j}^{\df}}{\ell_{j}+1}\stirlmat{\ell_{j}+1}\powermat{\df}{\ell_{j}+1}\mult[\ell_{j}+1]^{-1}\pcoeffbl[j]\right\} _{j=1}^{\np}.
\end{eqnarray*}

\end{prop}
\begin{minipage}[t]{1\columnwidth}%
\end{minipage}
\begin{prop}
\label{prop:conf-jac-fact}The Jacobian matrix of the map $\cpm$
can be decomposed as follows:
\begin{eqnarray*}
\jac_{\cpm} & = & \cvand_{\init,\df}\left(\jp_{1},\ell_{1}+1,\dots,\jp_{\np},\ell_{\np}+1\right)\diag\left\{ \ccoeffbl[j]\right\} _{j=1}^{\np}\\
 & = & \cpvand_{\init,\df}\left(\jp_{1},\ell_{1}+1,\dots,\jp_{\np},\ell_{\np}+1\right)\diag\left\{ \stirlmat{\ell_{j}+1}^{-1}\powermat{\jp_{j}^{-1}}{\ell_{j}+1}\ccoeffbl[j]\right\} _{j=1}^{\np}\\
 & = & \cvand_{\df}^{\#}\left(\jp_{1},\ell_{1}+1,\dots,\jp_{\np},\ell_{\np}+1\right)\diag\left\{ \jp_{j}^{t}\powermat{\jp_{j}^{\df}}{\ell_{j}+1}\stirlmat{\ell_{j}+1}\powermat{\df}{\ell_{j}+1}\mult[\ell_{j}+1]^{-1}\stirlmat{\ell_{j}+1}^{-1}\powermat{\jp_{j}^{-1}}{\ell_{j}+1}\ccoeffbl[j]\right\} _{j=1}^{\np}.
\end{eqnarray*}

\end{prop}

\subsection{Auxiliary lemma}

We'll need the following elementary computation.
\begin{lem}
\label{lem:norm-product-upper-tr}Let $A_{1,},\dots,A_{k},B$ be $n\times n$
upper triangular matrices over $\complexfield$, and $\vec c=\left(c_{1},\dots,c_{n}\right)^{T}\in\complexfield^{n}$
some $n\times1$ vector. Denote the entries of $A_{\ell}$ by $\left(a_{i,j}^{\left(\ell\right)}\right)_{i,j=1,\dots,n}$,
and those of $B$ by $\left(b_{i,j}\right)$. Let the vector $\vec d\in\complexfield^{n}$
be defined as
\begin{equation}
\vec d=BA_{k}A_{k-1}\cdots A_{1}\vec c.\label{eq:prod-chain}
\end{equation}
Fix $1\leqslant j\leqslant n$. Assume that exactly $0\leqslant m\leqslant k$
matrices among the $A_{1},\dots,A_{k}$ are strictly diagonal. Then
the $j-$th component of the vector $\vec d$ satisfies
\[
\left|d_{j}\right|\leq\left(n-j\right)^{k-m}\max_{i\geqslant j}\left\{ \left|c_{i}\right|\right\} \left(\sum_{i\geqslant j}\left|b_{j,i}\right|\right)\left(\prod_{\ell=1}^{k}\alpha_{j}^{\left(\ell\right)}\right),
\]
where
\[
\alpha_{j}^{\left(\ell\right)}=\max_{i,r\geqslant j}\left|a_{i,r}^{\left(\ell\right)}\right|.
\]
\end{lem}
\begin{proof}
By induction on $k$. Consider first $\vec d=B\vec c,$ the conclusion
is obvious. In the induction step, take $\tilde{B}=BA_{k}$. If $A_{k}$
is strictly diagonal, then the corresponding entries of $\tilde{B}$
satisfy
\[
\sum_{i\geqslant j}\left|\tilde{b}_{j,i}\right|\leqslant\left(\sum_{i\geqslant j}\left|b_{j,i}\right|\right)\max_{i\geqslant j}\left|a_{i,i}^{\left(k\right)}\right|.
\]

In the general case, both $B,A_{k}$ are upper triangular, and so
clearly
\[
\sum_{i\geqslant j}\left|\tilde{b}_{j,i}\right|\leqslant\left(n-j\right)\left(\sum_{i\geqslant j}\left|b_{j,i}\right|\right)\max_{i,r\geqslant j}\left|a_{i,r}^{\left(k\right)}\right|.
\]
This proves the claim for $k$.
\end{proof}

\subsection{\label{sub:proof-poly}Proof of \prettyref{thm:polynomial-stability}}

The proof of \prettyref{thm:polynomial-stability} boils down to estimating
the corresponding entries in the vector
\begin{equation}
\underbrace{\left[\begin{array}{c}
\Delta\jc_{0,1}\\
\vdots\\
\Delta\jc_{\ell_{1}-1,1}\\
\Delta\jp_{1}\\
\vdots
\end{array}\right]}_{\vec{\Delta x}}\sim\jac_{{\ppm}^{-1}}\underbrace{\left[\begin{array}{c}
O\left(\err\right)\\
\vdots\\
O\left(\err\right)\\
\\
\vdots
\end{array}\right]}_{\Delta\vec m}.\label{eq:errors-vec-to-eval}
\end{equation}

From \prettyref{prop:poly-jac-fact} we immediately get:
\[
\jac_{{\ppm}^{-1}}=\diag\left\{ \underbrace{\jp_{j}^{-\init}\pcoeffbl^{-1}\mult[l_{j}+1]\powermat{\df^{-1}}{\ell_{j}+1}\stirlmat{\ell_{j}+1}^{-1}\powermat{\jp_{j}^{-\df}}{\ell_{j}+1}}_{=B_{j}}\right\} _{j=1}^{\np}\left(\cvand_{\df}^{\#}\right)^{-1}.
\]

Now we represent \eqref{eq:errors-vec-to-eval} in the form \eqref{eq:prod-chain}.

First denote
\[
\vec w=\left(\cvand_{\df}^{\#}\right)^{-1}\Delta\vec m.
\]
Fix $j=1\dots,\np$ and some $\ell=0,\dots,\ell_{j}$. Let $\gamma\left(j,\ell\right)=\ell_{1}+1+\dots+\ell_{j-1}+1+\ell+1$.
If $u_{j,\ell}$ is the $\gamma$-th row of $\left(\cvand_{\df}^{\#}\right)^{-1}$,
then
\[
\left|w_{\gamma}\right|\leq\|u_{j,\ell}\|_{1}\err.
\]

Finally, notice that all the blocks $B_{j}=\left(b_{m,n}^{\left(j\right)}\right)$
in the big matrix multiplying $\left(\cvand_{\df}^{\#}\right)^{-1}$
are \emph{upper triangular}.

Then we write
\begin{equation}
\left[\begin{array}{c}
\Delta\jc_{0,j}\\
\vdots\\
\Delta\jc_{\ell_{j}-1,j}\\
\Delta\jp_{j}
\end{array}\right]=\jp_{j}^{-\init}E_{j}^{-1}\mult[l_{j}+1]\powermat{\df^{-1}}{\ell_{j}+1}\stirlmat{\ell_{j}+1}^{-1}\powermat{\jp_{j}^{-\df}}{\ell_{j}+1}\left[\begin{array}{c}
w_{\gamma\left(j,0\right)}\\
w_{\gamma\left(j,1\right)}\\
\vdots\\
w_{\gamma\left(j,\ell_{j}\right)}
\end{array}\right].\label{eq:errors-as-chain-poly}
\end{equation}

and it is left to evaluate each entry of the vector in the left-hand
side using \prettyref{lem:norm-product-upper-tr}.

Recall that $\left|\jp_{j}\right|=1$. For the last entry we clearly
have
\[
\left|\Delta\jp_{j}\right|\leqslant\frac{1}{\left|\jc_{\ell_{j}-1,j}\right|}\cdot1\cdot\frac{1}{\df^{\ell_{j}}}\cdot1\cdot\left|w_{\gamma\left(j,\ell_{j}\right)}\right|=\frac{1}{\df^{\ell_{j}}\left|\jc_{\ell_{j}-1,j}\right|}\|u_{j,\ell_{j}}\|_{1}\err.
\]

Combining this with \eqref{eq:cvand-norm-bound-2} proves the claim
for $\left|\Delta\jp_{j}\right|$.

In order to evaluate $\left|\Delta\jc_{\ell,j}\right|$ for $\ell=0,\dots,\ell_{j}-1,$
note that only $\pcoeffbl^{-1},\;\mult[\ell_{j}+1]$ and $\stirlmat{\ell_{j}+1}^{-1}$
are non-diagonal. Therefore by \prettyref{lem:norm-product-upper-tr},
\eqref{eq:pcoeffbl-inv} and \eqref{eq:errors-as-chain-poly}
\begin{eqnarray*}
\left|\Delta\jc_{\ell,j}\right| & \leq & \left(\ell_{j}-\ell\right)^{3}\underbrace{\max_{i\geq\ell}\left|w_{\gamma\left(j,i\right)}\right|}_{=\|u_{j,\ell}\|_{1}\err}\underbrace{\left(1+\frac{\left|\jc_{\ell-1,j}\right|}{\left|\jc_{\ell_{j}-1,j}\right|}\right)}_{\sum_{i\geqslant\ell}\left|\left(\pcoeffbl[j]^{-1}\right)_{\ell,i}\right|}\times\\
 &  & \times\underbrace{\max\left\{ 1,{\ell_{j}-1 \choose \ell_{j}-\ell}\init^{\ell_{j}-\ell}\right\} }_{\max_{r,s\geq\ell}\left|\left(\mult[l_{j}+1]\right)_{r,s}\right|}\df^{-\ell}\underbrace{\max\left\{ 1,\stf{\ell_{j}}{\ell}\right\} }_{\max_{r,s\geqslant\ell}\left|\left(\stirlmat{\ell_{j}+1}^{-1}\right)_{r,s}\right|},
\end{eqnarray*}
which proves the claim with
\begin{equation}
C_{1}\left(\ell,\ell_{j}\right)=\frac{2}{\ell!}\left(\ell_{j}-\ell\right)^{3}\max\left\{ 1,{\ell_{j}-1 \choose \ell_{j}-\ell}\right\} \max\left\{ 1,\stf{\ell_{j}}{\ell}\right\} .\label{eq:cij-l-lj-def}
\end{equation}

\subsection{Proof of \prettyref{thm:confluent-stability}}

Proceeding exactly as in \prettyref{sub:proof-poly}, we obtain instead
of \eqref{eq:errors-as-chain-poly} the following identity:
\begin{equation}
\left[\begin{array}{c}
\Delta\jc_{0,j}\\
\vdots\\
\Delta\jc_{\ell_{j}-1,j}\\
\Delta\jp_{j}
\end{array}\right]=\jp_{j}^{-\init}\ccoeffbl[j]^{-1}\powermat{\jp_{j}}{\ell_{j}+1}\stirlmat{\ell_{j}+1}\mult[l_{j}+1]\powermat{\df^{-1}}{\ell_{j}+1}\stirlmat{\ell_{j}+1}^{-1}\powermat{\jp_{j}^{-\df}}{\ell_{j}+1}\left[\begin{array}{c}
w_{\gamma\left(j,0\right)}\\
w_{\gamma\left(j,1\right)}\\
\vdots\\
w_{\gamma\left(j,\ell_{j}\right)}
\end{array}\right].\label{eq:errors-as-chain-conf}
\end{equation}

For the last entry in the left-hand side we have by \prettyref{lem:norm-product-upper-tr}
\begin{align*}
\left|\Delta\jp_{j}\right| & \leqslant\frac{1}{\left|\jp_{j}\right|^{\init}}\cdot\frac{1}{\left|\jc_{\ell_{j}-1,j}\right|}\cdot\left|\jp_{j}\right|^{\ell_{j}}\cdot1\cdot1\cdot\frac{1}{\df^{\ell_{j}}}\cdot1\cdot\frac{1}{\left|\jp_{j}\right|^{\df\ell_{j}}}\left|w_{\gamma\left(j,\ell_{j}\right)}\right|\\
 & =\left|\jp_{j}\right|^{\ell_{j}-\init-\df\ell_{j}}\frac{1}{\left|\jc_{\ell_{j}-1,j}\right|\df^{\ell_{j}}}\|u_{j,\ell_{j}}\|_{1}\err\\
 & \leqslant\frac{\left|\jp_{j}\right|^{\ell_{j}-\init-\df\ell_{j}}}{\left|\jc_{\ell_{j}-1,j}\right|\df^{\ell_{j}}}\left(\frac{2}{\delta_{\df}}\right)^{\nparams}\frac{2}{\ell_{j}!}\err,
\end{align*}
proving the claim for $\left|\Delta\jp_{j}\right|$. 

In order to evaluate $\left|\Delta\jc_{\ell,j}\right|$ for $\ell=0,\dots,\ell_{j}-1,$
note that only $\ccoeffbl[j]^{-1},\;\stirlmat{\ell_{j}+1},\;\mult[\ell_{j}+1]$
and $\stirlmat{\ell_{j}+1}^{-1}$ are non-diagonal. Therefore by \prettyref{lem:norm-product-upper-tr},
\eqref{eq:pcoeffbl-inv} and \eqref{eq:errors-as-chain-conf} we have
\begin{eqnarray*}
\left|\Delta\jc_{\ell,j}\right| & \leq & \left|\jp_{j}\right|^{-\init}\left(\ell_{j}-\ell\right)^{4}\underbrace{\max_{i\geq\ell}\left|w_{\gamma\left(j,i\right)}\right|}_{=\|u_{j,\ell}\|_{1}\err}\underbrace{\left(1+\frac{\left|\jc_{\ell-1,j}\right|}{\left|\jc_{\ell_{j}-1,j}\right|}\right)}_{\sum_{i\geqslant\ell}\left|\left(\ccoeffbl[j]^{-1}\right)_{\ell,i}\right|}\left|\jp_{j}\right|^{\ell}\underbrace{\max\left\{ 1,\sts{\ell_{j}}{\ell}\right\} }_{\max_{r,s\geqslant\ell}\left|\left(\stirlmat{\ell_{j}+1}\right)_{r,s}\right|}\times\\
 &  & \times\underbrace{\max\left\{ 1,{\ell_{j}-1 \choose \ell_{j}-\ell}\init^{\ell_{j}-\ell}\right\} }_{\max_{r,s\geq\ell}\left|\left(\mult[\ell_{j}+1]\right)_{r,s}\right|}\df^{-\ell}\underbrace{\max\left\{ 1,\stf{\ell_{j}}{\ell}\right\} }_{\max_{r,s\geqslant\ell}\left|\left(\stirlmat{\ell_{j}+1}^{-1}\right)_{r,s}\right|}\left|\jp_{j}\right|^{-\df\ell_{j}}\|u_{j,\ell}\|_{1}\err\\
 & \leqslant & C_{2}\left(\ell,\ell_{j}\right)\left|\jp_{j}\right|^{\ell-\init-\df\ell_{j}}\left(1+\frac{\left|\jc_{\ell-1,j}\right|}{\left|\jc_{\ell_{j}-1,j}\right|}\right)\init^{\ell_{j}-\ell}\df^{-\ell}\left(\frac{2}{\delta_{\df}}\right)^{\nparams}\left(\frac{1}{2}+\frac{\nparams}{\delta_{\df}}\right)^{\ell_{j}-\ell}\err,
\end{eqnarray*}
with 
\begin{equation}
C_{2}\left(\ell,\ell_{j}\right)=\frac{2}{\ell!}\left(\ell_{j}-\ell\right)^{4}\max\left\{ 1,{\ell_{j}-1 \choose \ell_{j}-\ell}\right\} \max\left\{ 1,\stf{\ell_{j}}{\ell}\right\} \max\left\{ 1,\sts{\ell_{j}}{\ell}\right\} .\label{eq:cij-l-lj-def-1}
\end{equation}
This completes the proof.

\section{\label{sec:numerical}Numerical experiments}

Recalling \prettyref{rem:non-shifted-equiv-pow}, it is immediate
that all known solution methods for non-decimated systems are directly
transferrable to decimated setting. Indeed, taking any algorithm for
the standard Prony-type system, one just needs to make the modifications
described in \prettyref{alg:decimation-impl-existing}.

\begin{algorithm}
For simplicity we consider only the recovery of the nodes $\left\{ z_{j}\right\} $.
\begin{enumerate}
\item Obtain initial approximations to the nodes.
\item Choose the decimation parameter $\df$ such that $\delta_{\df}$ is
not too small.
\item Feed the original algorithm with the decimated measurements $\meas[0],\meas[\df],\meas[2\df]\dots,$
and obtain the estimated node $w_{j}$.
\item Take $z_{j}=\sqrt[\df]{w_{j}}$.
\end{enumerate}
\caption{Implementing decimation for existing algorithms}
\label{alg:decimation-impl-existing}
\end{algorithm}

We have tested the decimation technique according to \prettyref{alg:decimation-impl-existing}
on two well-known algorithms for Prony systems - generalized ESPRIT
\cite{badeau2008performance} and nonlinear least squares (LS, implemented
by MATLAB's \texttt{lsqnonlin}). In the first experiment, we fixed
the number of measurements to be 66, and changed the decimation parameter
$\df$, while keeping the noise level constant. The accuracy of recovery
increased with $\df$ -- see \prettyref{fig:fixed-number-meas}. In
the second experiment, we fixed the highest available measurement
to be $\sc=1600$, and changed the decimation from $\df=1$ to $\df=100$
(thereby reducing the number of measurements from $1600$ to just
$16$). The accuracy of recovery stayed relatively constant -- see
\prettyref{fig:full-decimation}. Note that such a reduction leads
to a corresponding decrease in the running time (calculating singular-value
decomposition of large matrices, as in ESPRIT, is a time-consuming
operation).

\begin{figure}[h]
\subfloat[ESPRIT, $\ord=3$]{\includegraphics[width=0.45\columnwidth]{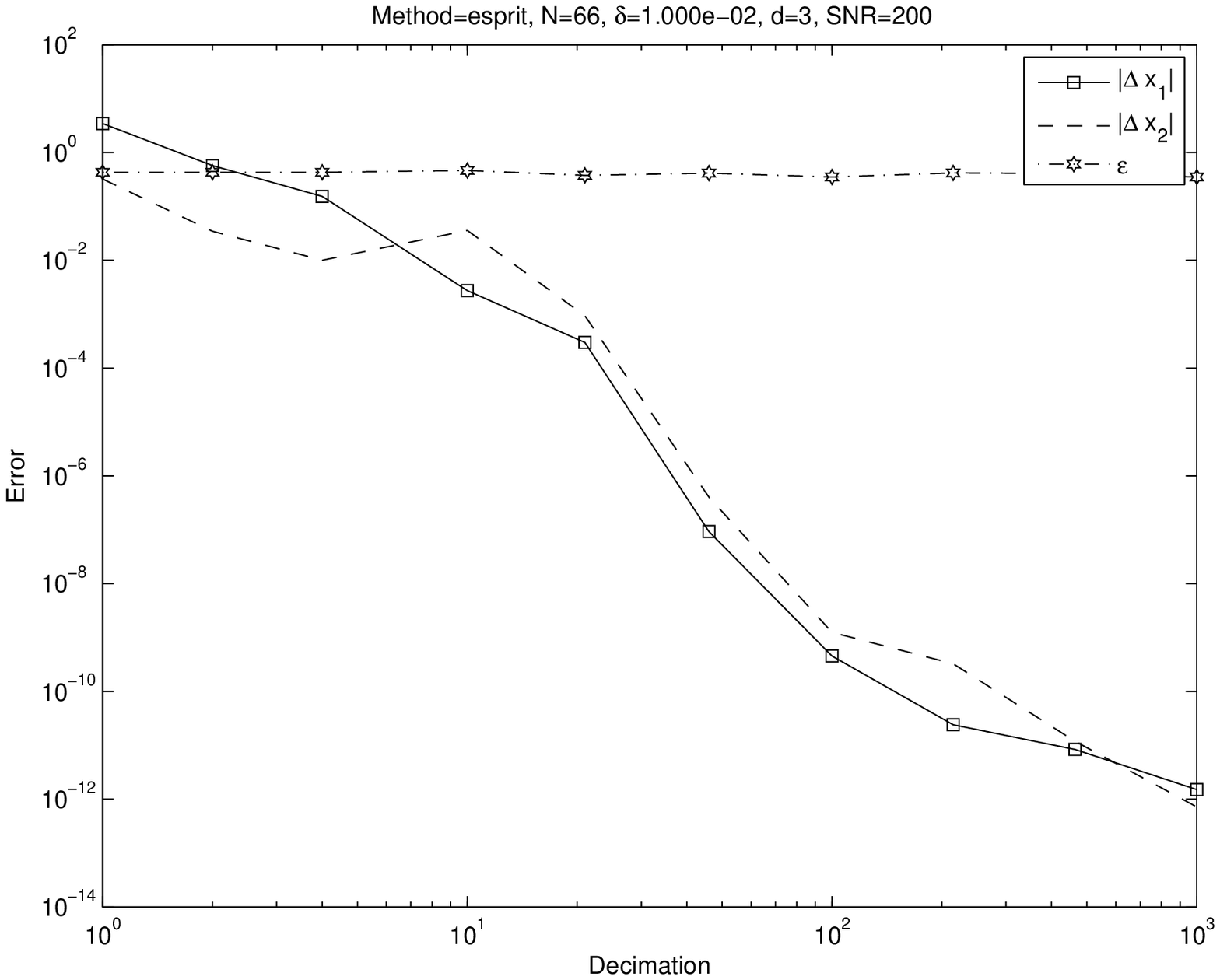}

}\subfloat[LS, $\ord=3$]{\includegraphics[width=0.45\columnwidth]{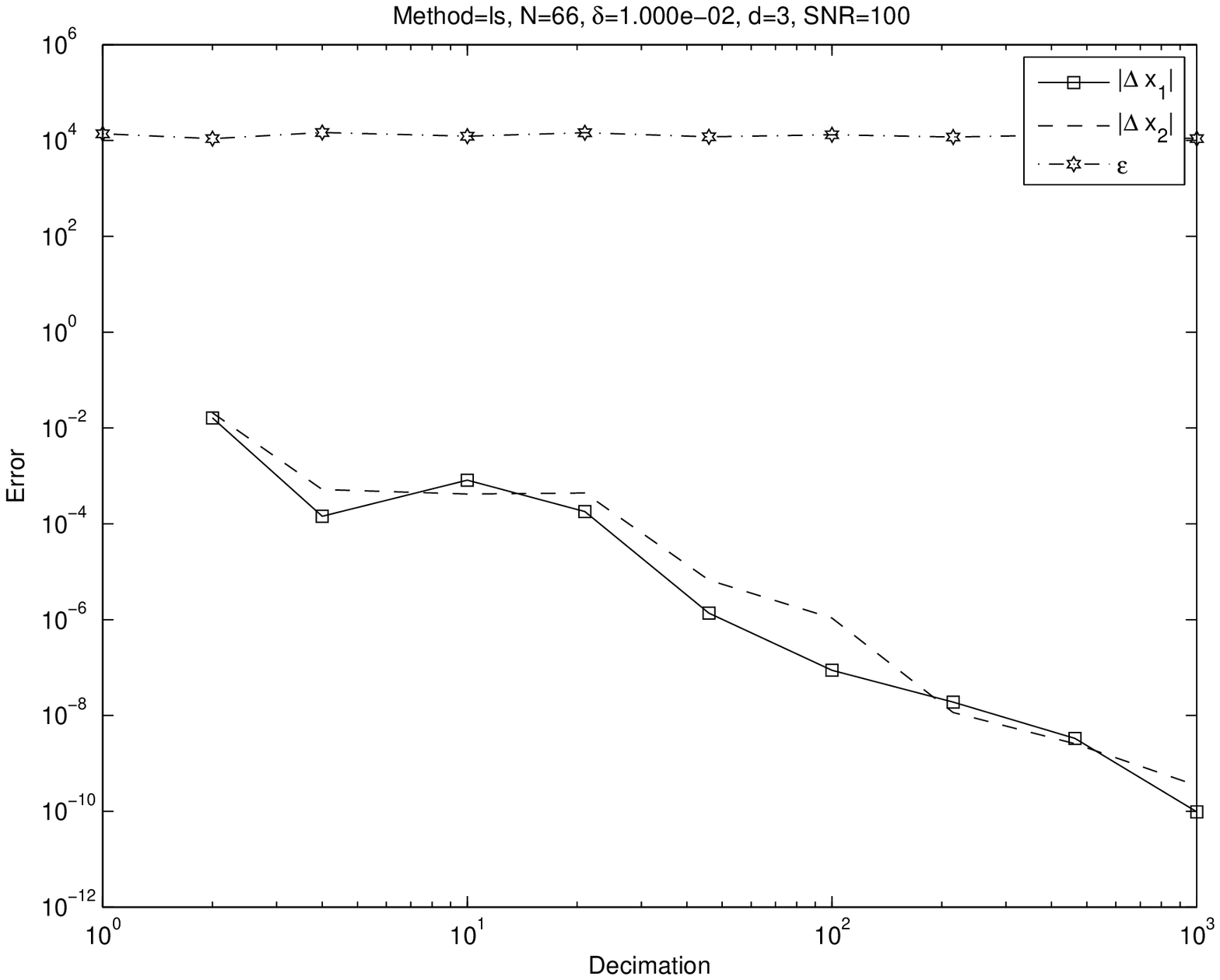}}

\caption[Decimation, fixed number of measurements]{Reconstruction error as a function of the decimation with fixed number
of measurements ($\sc=66$). The signal has two nodes with distance
$\delta=10^{-2}$ between each other. Notice that ESPRIT requires
significantly higher Signal-to-Noise Ratio in order to achieve the
same performance as LS.}
\label{fig:fixed-number-meas}
\end{figure}

\begin{figure}[h]
\subfloat[ESPRIT, $\ord=2$]{\includegraphics[width=0.45\columnwidth]{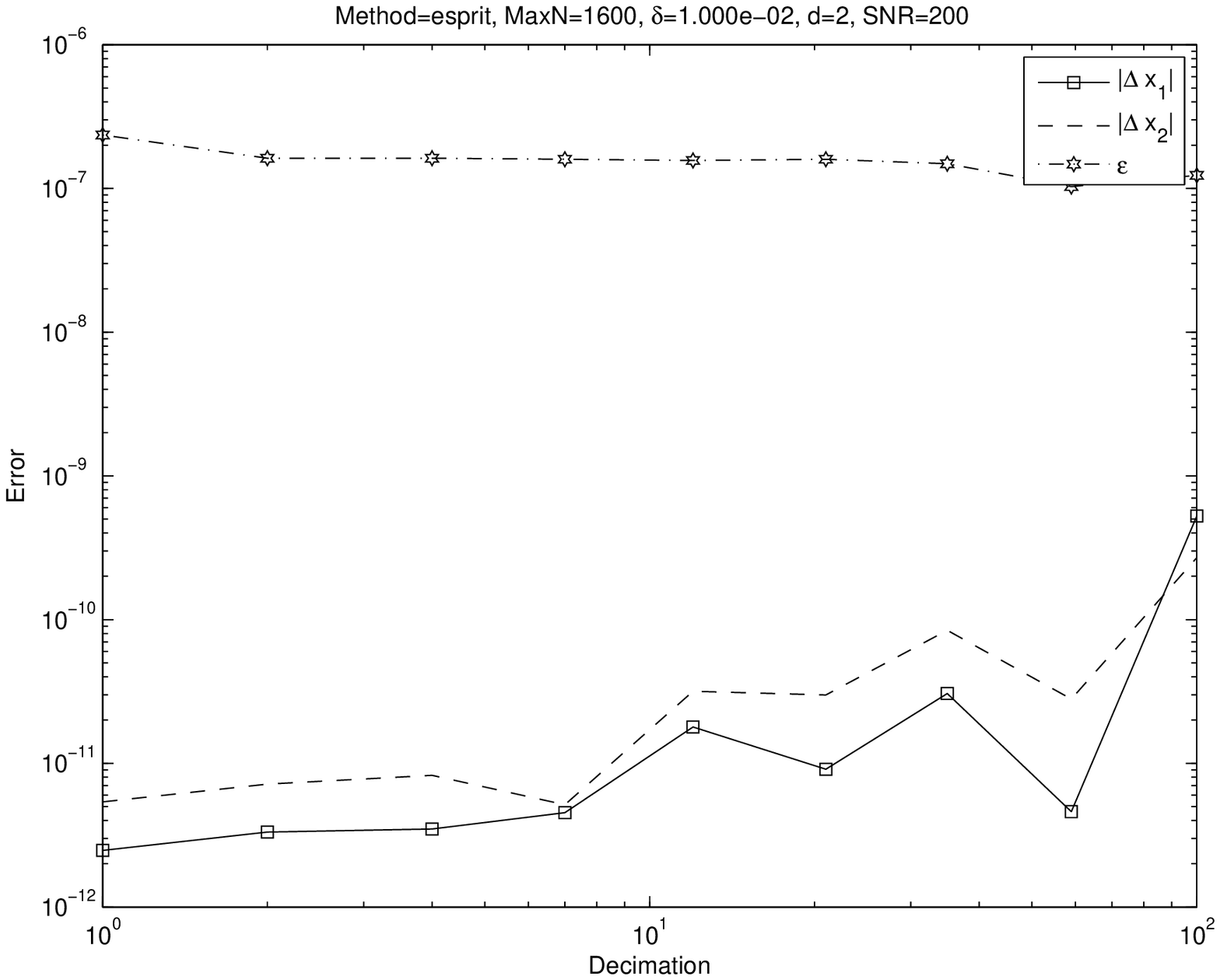}

} \subfloat[LS, $\ord=2$]{\includegraphics[width=0.45\columnwidth]{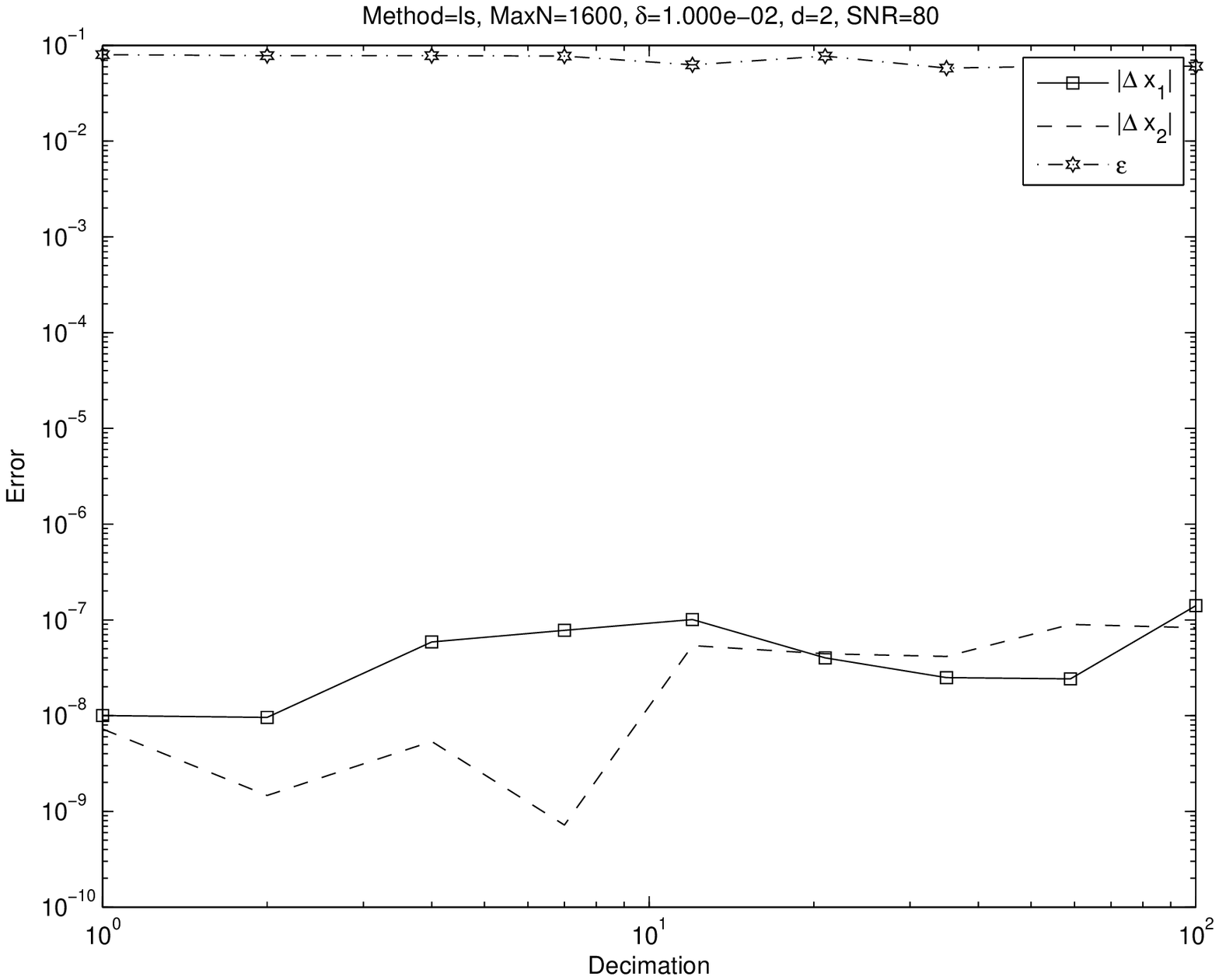}}

\caption[Decimation, reducing number of measurements]{Reconstruction error as a function of the decimation, reducing number
of measurements from $\sc=1600$ to $\sc=16$. The signal has two
nodes with distance $\delta=10^{-2}$ between each other. The reconstruction
accuracy remains almost constant.}
\label{fig:full-decimation}
\end{figure}

\section{\label{sec:relation-known-bounds}Relation to known lower bounds}

\subsection{\label{sub:crb-pace}Cramer-Rao Lower Bounds}

The polynomial Prony system \eqref{eq:polynomial-prony} is equivalent
to the so-called PACE (Polynomial-Amplitude-Complex-Exponentials)
model \cite{badeau2006high,badeau2008cramer} known from signal processing
literature. The Cramer-Rao bound (CRB) (which gives a lower bound
for the variance of any unbiased estimator, see \cite{kay1993fundamentals})
of the PACE model in white Gaussian noise is as follows (note that
the original expressions have been appropriately modified to match
the notations of this paper).
\begin{thm}[{\cite[Propositions III.1, III.3]{badeau2008cramer}}]
\label{thm:crb}Let the noise have variance $\sigma^{2}$, then for
small number of samples $\ns$ we have%
\footnote{Here $\mathfrak{R}\left(\cdot\right)$ denotes the real part.%
}
\begin{align}
CRB\left\{ z_{j}\right\}  & \approx\frac{\sigma^{2}}{\left|\jc_{\ell_{j}-1,j}\right|^{2}},\label{eq:crb-1-nodes}\\
CRB\left\{ \jc_{\ell,j}\right\}  & \approx\sigma^{2}\left(C_{1}\left|\frac{\jc_{\ell-1,j}}{\jc_{\ell_{j}-1,j}}\right|^{2}+C_{2}\mathfrak{R}\left\{ \frac{\jc_{\ell-1,j}}{\jc_{\ell_{j}-1,j}}\right\} +1\right),\qquad\ell=1,2,\dots,\ell_{j}-1.\label{eq:crb-1-ampl}
\end{align}
When the number of samples $\sc\gg1$, then the asymptotic CRB bounds
satisfy
\end{thm}
\begin{eqnarray}
CRB\left\{ z_{j}\right\}  & \approx & \frac{\sigma^{2}}{\left|\jc_{\ell_{j}-1,j}\right|^{2}\sc^{2\ell_{j}+1}},\label{eq:crb-2-nodes}\\
CRB\left\{ \jc_{\ell,j}\right\}  & \approx & \frac{\sigma^{2}}{\sc^{2\ell+1}},\qquad\ell=1,2,\dots,\ell_{j}-1.\label{eq:crb-2-ampl}
\end{eqnarray}

In the context of algebraic reconstruction, these bounds are not always
applicable, since they assume a particular statistical distribution
for the error terms $\err_{k}$ (see e.g. \prettyref{sec:fourier}).

However, we notice the clear analogy between the above bounds and
our main results. In particular, applying \prettyref{thm:polynomial-stability}
in the case $\init=0,\;\df=1$, we see that \eqref{eq:crb-1-nodes},
\eqref{eq:crb-1-ampl} and, respectively, \eqref{eq:poly-stability-nodes},
\eqref{eq:poly-stability-ampl} have the same qualitative dependence
on the parameters -- in particular, inverse proportionality w.r.t
the highest coefficient $\jc_{\ell_{j}-1,j}$. Furthermore, in the
asymptotic setting $\ns\gg1$, considering decimation with $\df=\frac{\ns}{\nparams}$,
we see that \eqref{eq:crb-2-nodes}, \eqref{eq:crb-2-ampl} and, respectively,
\eqref{eq:poly-stability-nodes}, \eqref{eq:poly-stability-ampl},
have similar asymptotic dependence on $\ns$. As we noted in \cite{batenkov2011accuracy},
the reason for such similarity is not a-priori clear (although it
could be partially attributed to the fact that both methods require
calculation of the partial derivatives of the measurements with respect
to the parameters), and it certainly prompts for further investigation.

\subsection{\label{sub:donoho-lower-bounds}Donoho's lower bounds}

\global\long\def\w{\omega}
\global\long\def\W{\Omega}
\global\long\def\e{\varepsilon}

Perhaps the most general lower bound for the performance of any method
whatsoever was given in the work of Donoho \cite{donoho1992superresolution}.
We consider this kind of bound to be very important, and elaborate
it further immediately below. 

Assume that one wants to recover a measure
\begin{equation}
\mu=\sum_{k=-\infty}^{\infty}a_{k}\delta_{k\Delta},\label{eq:donoho-measure}
\end{equation}
supported on a lattice $\left\{ k\Delta\right\} _{k=-\infty}^{\infty}$,
where $\Delta\ll1$, from noisy measurements
\begin{eqnarray*}
y\left(\w\right) & = & \widehat{\mu}\left(\w\right)+z\left(\w\right),\qquad\left|\w\right|\leq\W\\
\widehat{\mu}\left(\w\right) & = & \sum_{k=-\infty}^{\infty}a_{k}e^{-\imath k\w\Delta}.
\end{eqnarray*}
An essential requirement is that the frequency cutoff is much smaller
than the Nyquist threshold, $\W\ll\frac{\pi}{\Delta}$, while the
measure $\mu$ is assumed to be sparse in the following sense.
\begin{defn*}
Let $S\left(R,\Delta\right)$ denote the set of measures of the form
\eqref{eq:donoho-measure}, such that
\begin{enumerate}
\item The sequence $\left\{ a_{k}\right\} $ is in $\ell_{1}$;
\item the \emph{Rayleigh index }of $\supp\left(\mu\right)$, defined for
any set $S$ by
\[
R^{*}\left(S\right)\isdef\min\left\{ r:\; r\geqslant\sup_{t}\#\left(S\cap[t,t+r)\right)\right\} ,
\]
satisfies
\[
R^{*}\left(\supp\left(\mu\right)\right)\leqslant R.
\]

\end{enumerate}
\end{defn*}
\begin{minipage}[t]{1\columnwidth}%
\end{minipage}
\begin{defn*}
For every discrete measure $\mu$, denote
\begin{equation}
\|\mu\|_{2}=\left(\sum_{t\in supp\left(\mu\right)}\left|\mu\left(t\right)\right|^{2}\right)^{\frac{1}{2}}.\label{eq:discrete-norm-l2-donoho}
\end{equation}

\end{defn*}
\begin{minipage}[t]{1\columnwidth}%
\end{minipage}
\begin{defn*}
\label{def:modulus-continuity}The \emph{modulus of continuity }is
\[
\Lambda\left(\e,S\left(R,\Delta\right),\W\right)\isdef\sup\left\{ \|\mu_{1}-\mu_{2}\|_{2}:\;\mu_{1},\mu_{2}\in S\left(R,\Delta\right),\;\|\hat{\mu_{1}}-\hat{\mu_{2}}\|_{L_{2}\left[-\W,\W\right]}\leq\e\right\} .
\]

\end{defn*}
In other words, the modulus of continuity measures by how much \emph{an
ideal algorithm, }reconstructing the unknown measure $\mu$ from its
truncated Fourier transform $\hat{\mu}$\emph{ on $\left[-\W,\W\right]$,
}would deviate from the true answer, if the input is perturbed by
$\e$ (in the $L_{2}$ norm).\emph{ }Therefore, bounding the modulus
of continuity provides a \emph{stability bound} for recovery of sparse
measures from noisy data. The following result is proved in \cite{donoho1992superresolution}
using interpolation methods for entire functions.
\begin{thm*}[Superresolution theorem]
\label{thm:donoho}Let $\W>2\pi$ and $\Delta_{0}\in\left(0,1\right)$.
Then there exist constants $C_{1}\left(R,\W,\Delta_{0}\right)$ and
$C_{2}\left(R,\W\right)$ such that for every $\Delta<\Delta_{0}$ 

\begin{equation}
C_{1}\left(R,\W,\Delta_{0}\right)\left(\frac{1}{\Delta}\right)^{2R-1}\e\leq\Lambda\left(\e,S\left(R,\Delta\right),\W\right)\leq C_{2}\left(R,\W\right)\left(\frac{1}{\Delta}\right)^{2R+1}\e.\label{eq:donoho-main-result}
\end{equation}
As $\Delta_{0}\to0$, the lower bound can be approximated by 
\begin{equation}
\Lambda\left(\e\right)\gtrsim\sqrt{{4R-2 \choose 2R-1}\frac{\pi\left(4R-1\right)}{\W}}\left(\frac{1}{\W\Delta}\right)^{2R-1}\e,\label{eq:donoho-lower-bound-superres}
\end{equation}
and so the ratio $\frac{1}{\W\Delta}$ can be called the ``superresolution
factor''.

\end{thm*}
It is important to point out that the bounds of the \nameref{thm:donoho}
are theoretical, and no practical way to achieve them is known \cite{candes2012towards}.

Also, these bounds do not estimate the error in the locations of the
spikes since the norm \eqref{eq:discrete-norm-l2-donoho} measures
only the magnitudes of the linear coefficients $\left\{ \jc_{j}\right\} .$

Now we would like to provide a coarse estimate on the performance
of the algebraic method in this setting.

Consider the Prony system
\begin{eqnarray*}
f\left(x\right) & = & \sum_{j=1}^{R}a_{j}\delta\left(x-\nd_{j}\right),\\
\hat{f}\left(k\right) & = & \sum_{j=1}^{R}a_{j}e^{-\imath k\nd_{j}},\qquad\left|k\right|\leq\W,
\end{eqnarray*}
with actual measurements $m_{k}=\hat{f}\left(k\right)+z\left(k\right)$
and $\left|z\left(k\right)\right|\leq\e$. Define
\[
\Delta=\min_{i\neq j}\left|\ee^{-\imath\nd_{j}}-\ee^{-\imath\nd_{i}}\right|.
\]

First assume that we solve the Prony system taking only $k=0,1,\dots,2R-1$.
Then by \prettyref{thm:polynomial-stability} with $t=0,\;\df=1$
we have
\begin{equation}
\left|\Delta a_{j}\right|\leq C_{1}\left(\frac{2}{\Delta}\right)^{2R}\left(\frac{1}{2}+\frac{2R}{\Delta}\right)\e\leq C_{1}R4^{R+1}\left(\frac{1}{\Delta}\right)^{2R+1}\e.\label{eq:donoho-by-prony}
\end{equation}

Next, consider the case $\Delta\to0$. We can apply the decimation
technique: namely, take the samples
\[
k=0,\left\lfloor \frac{\W}{2R-1}\right\rfloor ,2\left\lfloor \frac{\W}{2R-1}\right\rfloor ,\dots,\W.
\]

The decimation parameter is therefore
\[
\df=\left\lfloor \frac{\W}{2R-1}\right\rfloor .
\]
Set $r_{0}=\frac{1}{2}$ for definiteness, and let $C_{2}\isdef\alpha\left(\frac{1}{2}\right)$.
Since $\Delta\to0$, we can assume that $\df\Delta<r_{0}$. By \prettyref{lem:prony-superresolution-deltap},
we have
\[
\left|\delta_{\df}\right|>C_{2}\Delta\df.
\]
Then by \prettyref{thm:polynomial-stability} the Prony accuracy is
bounded by
\begin{eqnarray}
\begin{split}\left|\Delta a_{j}\right| & \leqslant C_{3}\left(\frac{2}{C_{2}\cdot\Delta\cdot\frac{\W}{2R-1}}\right)^{2R}\left(\frac{1}{2}+\frac{2R}{C_{2}\cdot\Delta\cdot\frac{\W}{2R-1}}\right)\e\\
 & \leqslant C_{3}R\left(\frac{4}{C_{2}}\right)^{R+1}\left(2R-1\right)^{2R+1}\left(\frac{1}{\Delta\W}\right)^{2R+1}\e.
\end{split}
\label{eq:donoho2-by-prony}
\end{eqnarray}
The clear similarity between the estimates \eqref{eq:donoho-by-prony},
\eqref{eq:donoho2-by-prony} and \eqref{eq:donoho-main-result}, \eqref{eq:donoho-lower-bound-superres},
in particular their asymptotic growth rates as $\Delta\to0$, strongly
suggests that the algebraic method, via solving Prony type systems,
has the potential to achieve best possible results for superresolution.

\section{\label{sec:fourier}Application to piecewise-smooth Fourier reconstruction}

In this section we present an application of Prony decimation to the
recent solution of the so-called Eckhoff's conjecture, as elaborated
in \cite{batFullFourier,batyomAlgFourier}. Our goal here is simply
to point out the strong connection of this problem with the main results
of this paper.

Consider the problem of reconstructing an integrable function $\fun:\left[-\pi,\pi\right]\to\RR$
from a finite number of its Fourier coefficients \eqref{eq:fourier-coeffs-spiketrain-1}.
If $f$ is $C^{d}$ and periodic, then the truncated Fourier series
$\frsum\isdef\sum_{|k|=0}^{\sc}\fc(\fun)\ee^{\imath kx}$ approximates
$f$ with error at most $C\cdot\sc^{-d-1}$, which is optimal. If,
however, $f$ is not smooth even at a single point, the rate of accuracy
drops to only $\sc^{-1}$. This accuracy problem, also known as the
Gibbs phenomenon, is very important in applications, such as calculation
of shock waves in  PDEs. It has received much attention especially
in the last few decades - see e.g. a recent book \cite{jerriGibbs11}. 

The so-called ``algebraic approach'' to this problem, first suggested
by K.Eckhoff \cite{eckhoff1995arf}, is as follows. Assume that $\fun$
has $\np>0$ jump discontinuities $\left\{ \nd_{j}\right\} _{j=1}^{\np}$
, and $\fun\in C^{\ord}$ in every segment $\left(\nd_{j-1},\nd_{j}\right)$.
We say that in this case $\fun$ belongs to the class $PC\left(\ord,\np\right)$.
Denote the associated jump magnitudes at $\nd_{j}$ by $\jc_{\ell,j}\isdef\der{\fun}{\ell}(\nd_{j}^{+})-\der{\fun}{\ell}(\nd_{j}^{-}).$
Then write the piecewise smooth $\fun$ as the sum $\fun=\smooth+\sing$,
where $\smooth(x)$ is smooth and periodic and $\sing(x)$ is a piecewise
polynomial of degree $\ord$, uniquely determined by $\left\{ \nd_{j}\right\} ,\left\{ \jc_{\ell,j}\right\} $
such that it ``absorbs'' all the discontinuities of $\fun$ and its
first $\ord$ derivatives. In particular, the Fourier coefficients
of $\sing$ have the explicit form{\small{
\begin{equation}
\fc(\sing)=\frac{1}{2\pi}\sum_{j=1}^{\np}\ee^{-\imath k\nd_{j}}\sum_{\ell=0}^{\ord}(\imath k)^{-\ell-1}\jc_{\ell,j},\quad k=1,2,\dots.\label{eq:singular-fourier-explicit}
\end{equation}
}}For $k\gg1$, we have $\left|\fc\left(\sing\right)\right|\sim k^{-1}$,
while $\left|\fc\left(\smooth\right)\right|\sim k^{-\ord-2}$. Consequently,
Eckhoff suggested to pick large enough $k$ and solve the approximate
system of equations \eqref{eq:polynomial-prony} where $\meas=2\pi\left(\imath k\right)^{\ord+1}\fc\left(\fun\right)$,
$z_{j}=\ee^{-\imath\nd_{j}}$ and $c_{\ell,j}=\imath^{\ell}\jc_{\ord-\ell,j}$.
His proposed method of solution was to use the known values $\left\{ \meas\right\} _{k\in I}$
where
\begin{equation}
I=\left\{ \sc-\left(\ord+1\right)\np+1,\sc-\left(\ord+1\right)\np+2,\dots,\sc\right\} \label{eq:eckhoff-index-set}
\end{equation}
to construct an algebraic equation satisfied by the unknowns $\left\{ \nd_{1},\dots,\nd_{\np}\right\} $,
and solve this equation numerically. Based on some explicit computations
for $\ord=1,2;\;\np=1$ and large number of numerical experiments,
he conjectured that his method would reconstruct the jump locations
with accuracy $\sc^{-\ord-2}$, and the values of the function between
the jumps with accuracy $\sc^{-\ord-1}$.

Let us consider the problem in the framework of Prony type system
\eqref{eq:polynomial-prony}. The error term is of magnitude $\left|\err\right|\sim\sc^{-1}$.
The index set \eqref{eq:eckhoff-index-set} is just $I_{\init,\df}$
with $\init\sim\sc,\;\df=1$ (i.e. no decimation). Therefore, by \prettyref{thm:polynomial-stability}
we get accuracy only of order $\left|\Delta\nd_{j}\right|\sim\sc^{-1}$.
This is indeed confirmed by our numerical experiments in \cite{batFullFourier}.
In \cite{batyomAlgFourier} we have partially overcome this difficulty
by considering the Prony system \eqref{eq:singular-fourier-explicit}
whose order was only half the actual smoothness of the function. In
effect, this corresponded to providing the error terms with additional
structure, which eventually lead to some cancellations and improvement
of the estimate $\sc^{-1}$ to $\sc^{-\left\lfloor \frac{d}{2}\right\rfloor -2}$.

Now consider the decimated setting for this problem. By the above,
we can approximate each jump $\nd_{j}$ up to accuracy $\sc^{-\left\lfloor \frac{d}{2}\right\rfloor -2}$.
Set
\[
\scc=\left\lfloor \frac{\sc}{\left(\ord+2\right)\np}\right\rfloor .
\]
Now take the index set $I_{\init,\df}$ where $\init=\df=\scc$, i.e.
$I_{\scc,\scc}=\left\{ \scc,2\scc,\dots,\sc\right\} .$ As before,
$\left|\epsilon\right|\sim\sc^{-1}$, but now due to decimation we
get accuracy $\left|\Delta\nd_{j}\right|\sim N^{-\ord-1}N^{-1}\sim\sc^{-\ord-2}$,
and by \prettyref{rem:aliasing} the decimation can be applied.  In
\cite{batFullFourier} we have shown that, indeed, by ``decimating''
the algorithm of \cite{batyomAlgFourier} in a certain sense, we get
the full (best possible) accuracy $\sim\sc^{-d-2}$ for the jumps,
and accuracy $\sim\sc^{-\ord-1}$ for values between the jumps.

Thus decimation for Prony systems has ultimately provided the solution
to Eckhoff's conjecture.

\section{Discussion}

In addition to the subspace shifting of \cite{maravic2005sar}, the
decimation approach has another analogue in the literature -- the
so-called ``arithmetic progresssion sampling'', developed by A.Sidi
for convergence acceleration of Richardson extrapolation problems
in numerical analysis \cite{sidi2003practical}. It would be interesting
to make this connection more elaborate and precise.

An important drawback of the ``Prony map'' approach is that it cannot
handle oversampling, namely, considering Prony type systems for the
measurements $\left\{ \nn{\meas}\right\} _{k\in S}$ where $\left|S\right|>\nparams$.
It remains to be seen whether this drawback can be eventually overcome.

As our results show, when two nodes collide, the Prony system has
an algebraic singularity. It is possible that further study of these
singularities, started in \cite{byPronySing12} using divided differences,
may prove to have an important role for obtaining sharp bounds for
the Prony stability problem.

An important research problem connected with Prony systems is non-uniform
sampling. While formally the ``Prony map'' approach can be applied
in the case when the set $S$ contains arbitrary $\nparams$ integers
between $0$ and $\ns-1$, the explicit analysis of Jacobians seems
to be rather difficult. In \cite{Batenkov2011} we provided some estimates
using discrete version of Turan-Nazarov inequalities for exponential
polynomials from \cite{friedland2011observation}. In these inequalities,
a central role is played by a geometric invariant of the sampling
set $S$ connected to its metric entropy. It can be shown that this
invariant is minimal precisely for evenly spaced sampling sets, giving
another justification for the decimation procedure. We intend to provide
the details in a future publication.

\bibliographystyle{plain}
\bibliography{../../../bibliography/all-bib}

\end{document}